\documentclass[notitlepage,11pt,reqno]{amsart}
\usepackage[foot]{amsaddr}
\usepackage{amssymb,nicefrac,bm,upgreek,mathtools,verbatim}
\usepackage[final]{hyperref}
\usepackage[mathscr]{eucal}
\usepackage{dsfont}
\usepackage[normalem]{ulem}
\usepackage{amsopn,esint}
\usepackage[T1]{fontenc}
\usepackage[shortlabels]{enumitem}
\usepackage[T1]{fontenc}
\usepackage[utf8]{inputenc}
\usepackage{apptools}

\AtAppendix{\counterwithin{theorem}{section}} 

\usepackage[margin=1in]{geometry}

\newcommand{\stkout}[1]{\ifmmode\text{\sout{\ensuremath{#1}}}\else\sout{#1}\fi}
\hypersetup{
  colorlinks=true,
  citecolor=black,
  linkcolor=black,
  urlcolor = blue,
  anchorcolor = blue,
  frenchlinks=false,
  pdfborder={0 0 0},
  naturalnames=false,
  hypertexnames=false,
  breaklinks}
\usepackage{cite}
\usepackage[capitalize]{cleveref}

\crefname{equation}{}{}
\newtheorem{theorem}{Theorem}[section]
\newtheorem{corollary}[theorem]{Corollary}
\newtheorem{proposition}[theorem]{Proposition}
\newtheorem{lemma}[theorem]{Lemma}
\newtheorem*{theorem*}{Theorem}

\theoremstyle{definition}
\newtheorem{definition}[theorem]{Definition}
\newtheorem{remark}[theorem]{Remark}
\newtheorem*{remark*}{Remark}

\numberwithin{equation}{section} 

\newcommand{\eps}{{\varepsilon}}

\newcommand{\cGs}{{\mathbb G}}
\newcommand{\cHs}{{\mathbb H}}

\newcommand{\cFs}{{\mathbb F}}
\newcommand{\cDs}{{\mathbb D}}
\newcommand{\cPs}{{{\mathbb P}_s}}
\newcommand{\loc}{\mathrm{loc}}

\newcommand{\cS}{{\mathcal{S}}}

\newcommand{\RRN}{\mathds{R}^{N}}
\newcommand{\R}{\mathds{R}}
\newcommand{\N}{\mathds{N}}
\newcommand{\RR}{\mathds{R}}
\newcommand{\PP}{\mathbb{P}}

\DeclareMathOperator{\loglap}{L_{\Delta}}
\DeclareMathOperator{\PV}{\text{P.V.}}

\newcommand{\id}{\bm{I}}

\DeclareMathOperator*{\Log}{\text{Log}\,}

\DeclareMathOperator*{\diam}{diam}

\DeclareMathOperator*{\supp}{supp}

\def\YYint#1#2#3{{\setbox0=\hbox{$#1{#2#3}{\iint}$}
		\vcenter{\hbox{$#2#3$}}\kern-.50\wd0}}

\def\XXint#1#2#3{{\setbox0=\hbox{$#1{#2#3}{\int}$}
		\vcenter{\hbox{$#2#3$}}\kern-.50\wd0}}

\usepackage{xcolor}
\allowdisplaybreaks

\newcommand{\ttl}{\texorpdfstring{$s$}{s}\MakeUppercase{-harmonic functions in the small order limit}}
\begin{document}
\title[$s$-harmonic functions in the small order limit]
{\ttl}

\author{Sven Jarohs, Abhrojyoti Sen, and Tobias Weth}
\address{Goethe-Universit\"{a}t Frankfurt, Institut f\"{u}r Mathematik, Robert-Mayer-Str. 10,
D-60629 Frankfurt, Germany. Email: jarohs@math.uni-frankfurt.de, sen@math.uni-frankfurt.de, weth@math.uni-frankfurt.de}

\begin{abstract} 
  We study families $u_s$ of functions satisfying the equations $(-\Delta)^s u_s=0$, $s \in (0,1)$ in a smooth bounded open set $\Omega \subset \RRN$. The main purpose of this paper is twofold. First, we provide a detailed analysis of the asymptotics of these families in the zero order limit $s \to 0^+$. Second, we study the differentiability of $u_s$ as a function of $s$. Most of our results are devoted to the associated Poisson problem, where the family $u_s$ is determined by the exterior condition $u_s = g$ in $\R^N \setminus \Omega$ for some fixed function $g \in L^\infty(\R^N \setminus \Omega)$. Our results show that both the zero order asymptotics and the differentiability properties of $u_s$ can be expressed in terms of the logarithmic Laplacian of suitable extensions of $g$. This allows to deduce pointwise monotonicity properties of $u_s$ in the order parameter $s$ for a large class of functions $g$.
\end{abstract}

\maketitle
\tableofcontents
\section{Introduction and main results}


In recent years, nonlocal operators have attracted considerable attention in analysis and partial differential equations. Among these, the fractional Laplacian has emerged as a central object. Let $s\in (0, 1)$ be fixed, then the fractional Laplacian can be defined formally through its integral representation
\begin{align}\label{defn_frac_lap}
(-\Delta)^s u(x) &= c_{N, s}\, \text{P.V.}\int_{\RR^N} \frac{u(x)-u(y)}{|x-y|^{N+2s}} \, dy=c_{N, s} \lim_{\varepsilon \to 0}\int_{\RR^N\setminus B_{\varepsilon}(x)}\frac{u(x)-u(y)}{|x-y|^{N+2s}}\, dy,
\end{align}
where $c_{N, s}$ is a normalization constant given by
\begin{align}\label{defn_c(N, s)}
    c_{N, s}=\frac{4^s\Gamma\left(\frac{N}{2}+s\right)}{\Gamma(2-s)\pi^{\frac{N}{2}}}s(1-s).
\end{align}
The expression \eqref{defn_frac_lap} is in particular well-defined for $u\in C^{2}(\RR^N)\cap L^1_{s}(\RR^N)$, where $L^1_s$ is a weighted $L^1$-space defined by
\begin{align}\label{L1_s space}
L^1_s(\RR^N):=\Bigl\{u \in L^1_{\loc}(\RR^N)\,\, :\,\, \int_{\RR^N}\frac{|u(x)|}{\left(1+|x|\right)^{N+2s}}\, dx<\infty \Bigr\}
\end{align}
endowed with the norm
\begin{align}\label{L1_s space2}
u \mapsto    \|u\|_{L^1_s(\RR^N)}:=\int_{\RR^N}\frac{|u(x)|}{\left(1+|x|\right)^{N+2s}}\, dx.
\end{align}
We emphasize that \eqref{L1_s space}, \eqref{L1_s space2} can be defined for any $s\in \RR$ and we shall also use this space, in particular, for $s=0$.

In the present paper, we are concerned with functions $u \in L^1_{s}(\RR^N)$, which are $s$-harmonic in some bounded domain $\Omega \subset \RRN$ with $C^2$ boundary, i.e., which satisfy
\begin{equation}
  \label{eq:new-s-harmonic}
 (-\Delta)^s u =0 \qquad \text{in}\,\,\, \Omega.  
\end{equation}
Here, equation~(\ref{eq:new-s-harmonic}) is to be understood in distributional sense, i.e.,
\begin{equation}
  \label{eq:new-s-harmonic-distributional}
\int_{\Omega}u (-\Delta)^s\phi\,dx  =0 \qquad \text{for all $\phi \in C^\infty_c(\Omega)$.}
\end{equation}
It is well known  that (\ref{eq:new-s-harmonic-distributional}) implies $u \in C^\infty(\Omega)$  and, in particular, that (\ref{eq:new-s-harmonic}) holds pointwisely in $\Omega$ (see e.g. \cite[Chapter I.6]{L72}, \cite{BB99} or \cite[Theorem 12.17]{Nicola17} and the references in there). While most of the literature treats $s$ as a fixed parameter in $(0, 1)$, it is  natural to ask how the property of being $s$-harmonic depends on the parameter $s$. In particular, it is an interesting question of what happens to (families of) $s$-harmonic function $u_s$ in the limiting regimes $s \to 1^-$ and $s\to 0^+$. Here, the first limit corresponds to the case where the fractional Laplacian converges to the classical Laplace operator, and the \textit{small-order} limit, $s \to 0^+$, corresponds to the case where $(-\Delta)^s$ tends to the identity.

In this paper, we will be concerned with the second limit. We start our discussion with a first observation regarding the flatness of uniformly bounded $s$-harmonic functions $u_s$ in the limit $s \to 0^+$. 

\begin{theorem}\label{thm1.0}
  Let $s_0 \in (0,1)$ and $\Omega \subset \RRN$ be a bounded domain. Also, let $u_s \in L^\infty(\RR^N)$, $s \in (0,s_0]$ be $s$-harmonic in $\Omega$ and such that
  \begin{equation}
    \label{eq:them1.0-assumption}
  \|u_s\|_{L^\infty(\RR^N)} \le C \qquad \text{for all $s \in (0,s_0]$ with a constant $C>0$.}
  \end{equation}
  Then we have that
  $$
  \inf_{c \in \RR}|u_s-c| \to 0 \qquad \text{as $s \to 0^+$ locally uniformly in $\Omega$.}
  $$
\end{theorem}

The proof of Theorem~\ref{thm1.0} relies on an explicit local Poisson representation formula for $s$-harmonic functions in balls. More precisely, if $u \in L^1_{s}(\RR^N)$ is $s$-harmonic in $B_r(x)$ for some $x \in \R^N$, $r>0$ and bounded in a neighborhood of $\overline{B_r(x)}$, then we have 
\begin{equation}
  \label{eq:poisson-rep-balls}
u(x) = \int_{\RR^N\setminus B_r(x)}P_{s,r}(x, y) u(y)\, dy \qquad \text{for $x \in B_r(x)$}
\end{equation}
with the associated Poisson kernel
\begin{align*}
    P_{s, r}(x, y)=\frac{\gamma_{N, s}}{|x-y|^N}\left(\frac{r^2-|x|^2}{|y|^2-r^2}\right)^s, \qquad \text{where}\quad \gamma_{N,s}=\frac{\Gamma(\frac{N}{2})}{\pi^{\frac{N}{2}}\Gamma(s)\Gamma(1-s)}.
\end{align*}
Using this representation, the convergence for $s\to 1^-$ of $s$-harmonic functions in $\Omega$ to \textit{classical harmonic} functions has been discussed in \cite[Theorem 5.80]{GH24}.
If we assume in addition that $u$ is nonnegative, we also have the following $s$-independent version of Harnack's inequality:
\begin{equation}\label{harnack s-0}
2^{-4N-2}\leq \frac{u(x_1)}{u(x_2)}\leq 2^{4N+2}\qquad\text{for all $x_1,x_2\in B_{r/4}(x)$.}
\end{equation}
This inequality is well-known to experts and follows easily from (\ref{eq:poisson-rep-balls}). For the convenience of the reader, we include the short proof in the appendix. 

While Theorem~\ref{thm1.0} implies that (uniformly bounded) $s$-harmonic functions become flatter and flatter in $\Omega$ as $s \to 0^+$, it clearly does not imply that $u_s$ converges locally to a fixed constant in $\Omega$. To analyze the shape of $u_s$ for $s$ near zero in greater detail, we now restrict our attention to $s$-harmonic functions $u_s$ satisfying $u_s \equiv g$ in $\R^N \setminus \Omega$ for a {\em fixed} (s-independent) function $g \in L^\infty(\R^N \setminus \Omega)$. Hence we consider the family of nonhomogeneous Dirichlet problems
\begin{equation}\label{main eq}
\begin{cases}
    (-\Delta)^s u =0\,\,\, &\text{in}\,\,\, \Omega\\
    u=g\,\,\, &\text{in}\,\,\, \RRN\setminus\Omega,
\end{cases}
\end{equation}
where $\Omega \subset \RRN$ is a bounded domain with $C^2$ boundary, $s\in (0, 1)$, and $N\geq 2$. For each fixed $s$, this problem has a unique solution $u_s \in L^\infty(\R^N)$, which can again be represented in terms of the associated Possion kernel as
\begin{equation}\label{intro-def-poisson}
    u_s=\PP_sg:=\int_{\RR^N\setminus \Omega}P_s(\cdot, y) g(y)\, dy\, \chi_{\Omega} +g\, \chi_{\RR^N\setminus\Omega},
\end{equation}
see e.\,g. \cite{BKK08}. In the following, we analyze the solution map 
\[
s \mapsto u_s
\]  
and study the behavior of $u_s$ as $s\to 0$ with respect to the exterior data $g$. In particular, our study is motivated by the following questions:\\[-0.2cm] 
\begin{center}
\begin{minipage}{14cm}
 \begin{enumerate}
    \item[{\bf Question 1.}] Under what conditions on the data $g$ does the family $(u_s)_s$ converge in $\Omega$ as $s \to 0^+$?\\[-0.2cm] 
    \item[{\bf Question 2.}] Is the solution map $s \mapsto u_s$ differentiable at $s=0$ under appropriate assumptions on the exterior data $g$? If so, can it be expressed in terms of the logarithmic Laplacian?\\[-0.2cm] 
    \item[{\bf Question 3.}] How does the dependence on $s$ behave for general $s \in (0,1)$, and are there some monotonicity properties with respect to $s$?\\[-0.2cm] 
\end{enumerate}   
\end{minipage}
\end{center}
Our main results answer these questions. For Question 1, we establish conditions on the exterior data $g$ ensuring almost uniform convergence of $u_s$ as $s \to 0^+$. We will see that, in particular, if the spherical averages of $g$ converge to a constant at infinity, then the solutions $u_s$ converge to this constant inside $\Omega$ as $s \to 0^+$. Moreover, we shall identify a first order term in the expansion of $u_s$ as $s \to 0^+$. Before we state the main theorem concerning Question 1 and Question 2, we recall the following definition from \cite{JSW20}.

\begin{definition}[Almost uniform convergence]
Let $\Omega \subset \RRN$ be an open bounded set, and let $u_n : \Omega \to \RR$ be a sequence of functions. We say that $(u_n)_n$ converges \textit{almost uniformly in $\Omega$} to a function $u$ if 
\begin{align*}\label{almost uniform def}
    \lim_{n \to \infty}\left\|\delta^{\varepsilon}_{\Omega}(u_n-u)\right\|_{L^{\infty}(\Omega)}=0\,\,\,\, \text{for every}\,\,\, \varepsilon >0,
\end{align*}
where, here and in the following, we put
$$
\delta_{\Omega}(x)=\text{dist} (x, \partial \Omega) \qquad \text{for $x \in \RRN$,}
$$
so $\delta_{\Omega}: \RR^N \to \RR$ is the boundary distance function.
\end{definition}

We introduce some standard notation. Let $B_r:= B_r(0)$ resp. $S_r:=S_r(0)$ denote the ball and sphere of radius $r>0$ centered at zero and let $\omega_{N-1}$ denote the $(N-1)$-dimensional volume of the unit sphere $S_{N-1}$. Moreover, for a function $g \in L^1(S_r)$, we let 
$$
\fint_{S_r}g(y)\,d\sigma(y):= \frac{1}{r^{N-1}\omega_{N-1}}\int_{S_r}g(y)\,d\sigma(y)
$$
denote the average integral of $g$ over $S_r$.

\begin{theorem}\label{thm1.1}
  Let $N\geq 2$, let $\Omega \subset \RRN$ be an open bounded set with $C^2$-boundary, and let $g \in L^{\infty}(\RR^N\setminus \Omega).$
  Moreover, let
  \begin{equation}
    \label{eq:def-g-tilde}
    \tilde g(r):= \fint_{S_r}g(y)\,d\sigma(y),
  \end{equation}
  which is well-defined for a.\,e. $r>0$ with $S_r \subset \R^N \setminus \Omega$.
  \begin{enumerate}
  \item[(i)] If $c_g := \lim \limits_{r \to +\infty}\tilde g(r)$ exists, then $u_s \to c_g$ almost uniformly in $\Omega$ as $s \to 0^+.$
  \item[(ii)] If $c_g := \lim \limits_{r \to +\infty}\tilde g(r)$ exists and 
    \begin{equation}
      \label{eq:first-order-exp-assumption}
      \int_{1}^\infty \frac{|\tilde g(r)-c_g|}{r}dr < \infty,
    \end{equation}
    then we have
    \begin{equation}
      \label{eq:first-order-exp}
    u_s = c_g + s \Bigl(L_g+ o(1)\Bigr) \qquad \text{in $\Omega$,}
    \end{equation}
    where
    \begin{equation}
      \label{eq:def-L-g}
    L_g(x)= c_{N} P.V.\int_{\RR^N \setminus \Omega}\frac{g(y)-c_g}{|x-y|^N}\,dy \qquad \text{for $x \in \Omega$}
    \end{equation}
 with $c_N=\Gamma(\frac{N}{2})\pi^{-\frac{N}{2}}=\frac{2}{\omega_{N-1}}$ and $o(1) \to 0$ almost uniformly in $\Omega$ as $s \to 0^+$.     
  \end{enumerate} 
\end{theorem}

\begin{remark}
  \label{loglap-def}
  (i)  We clarify that the principle value integral in (\ref{eq:def-L-g}) is concerned with the behavior at infinity, i.e.,
    \begin{equation}
      \label{eq:def-L-g-pv-sense}
    L_g(x)= c_{N} \lim_{r \to \infty}\int_{B_r \setminus \Omega}\frac{g(y)-c_g}{|x-y|^N}\,dy \quad \text{for $x \in \Omega$.}
  \end{equation}
  Under the more restrictive assumption that $g-c_g$ extends to a function in $L^1_0(\RR^N)$, this integral also exists in Lebesgue sense. In contrast, the weaker assumption (\ref{eq:first-order-exp-assumption}) allows large oscillations of $g$ in the spherical variable. For a proof of the fact that
  the limit in (\ref{eq:def-L-g-pv-sense}) exists in the general case of functions $g \in L^\infty(\RR^N \setminus \Omega)$ satisfying (\ref{eq:first-order-exp-assumption}), see Lemma~\ref{regarding-def-L-g-pv-sense} below.

  (ii) The function $L_g$ can be seen as the restriction of $-L_\Delta (g1_{\RR^N \setminus \Omega}-c_g)$ to $\Omega$, where $L_\Delta$ denotes the logarithmic Laplacian. The operator $L_\Delta$ has been introduced in \cite{CW19}, and it appears in various contexts where the zero order limit of fractional problems is considered, see for instance \cite{FJW22,CW19,CHW,AS23,HS22,DLNP21,FS25, CLS24, HSLRS25, CV22, C23}. For the definition and properties of logarithmic Laplacian, see Section~\ref{priliminary} below.

  (iii) The first order expansion~(\ref{eq:first-order-exp}) shows that pointwisely we have
  $$
  L_g = \frac{d}{ds}\Bigl|_{s=0}u_s \qquad \text{in $\Omega$.}
  $$
  In the case where $g \in L^{\infty}(\RR^N\setminus \Omega)$ is nonnegative and $c_g =0$, we  deduce from \eqref{harnack s-0} and \cref{thm1.1}\, (ii) that
  for every ball $B_r(x)\subset \Omega$ we have 
$$
2^{-4N-2}\leq \frac{L_{\Delta}(1_{\RR^N\setminus \Omega}g)(x_1)}{L_{\Delta}(1_{\RR^N\setminus \Omega}g)(x_2)} \le 2^{4N+2} \qquad \text{for all $x_1,x_2 \in B_{r/4}(x)$.}
$$
By approximation, this also holds for nonnegative functions $g \in L^{\infty}_{\loc}(\RR^N\setminus \Omega) \cap L^1_0(\R^N)$.
\end{remark}

As a consequence of Theorem~\ref{thm1.1}, we derive the following limiting Harnack inequality.

\begin{corollary}\label{cor-harnack}
  Let $N\geq 2$, $\Omega \subset \RRN$ be an open bounded set with $C^2$-boundary, let $g \in L^{\infty}(\RR^N\setminus \Omega)$ satisfy $c_g = \lim \limits_{r \to \infty}\tilde g(r)=0$ with $\tilde g$ defined in (\ref{eq:def-g-tilde}), and let $u_s$ denote the solutions of \eqref{main eq} for $s \in (0,1)$. Then we have
\begin{align}\label{harnack}
    \lim_{s\to 0}\frac{u_s(x)}{s}\leq \Bigl(\frac{\diam(\Omega)}{\delta_\Omega(x)}\Bigr)^N \Bigg(\,\lim_{s\to 0}\frac{u_s(y)}{s}+c_N \int_{\RR^N\setminus \Omega}\frac{g^{-}(z)}{|y-z|^N}\,dz\Bigg) \quad \text{for $x,y \in \Omega$.}
\end{align}
\end{corollary}

We emphasize that if $g\geq 0, $ then \eqref{harnack} reduces to
\begin{align*}
 \lim_{s\to 0}\frac{u_s(x)}{s}\leq \Bigl(\frac{\diam(\Omega)}{\delta_\Omega(x)}\Bigr)^N \lim_{s\to 0}\frac{u_s(y)}{s}\quad\text{for all $x,y\in \Omega$.}
\end{align*}

Next, in view of Question 1 we also show that, if $g$ is as in \cref{thm1.1} but $\lim \limits_{r \to \infty}\tilde g(r)$ does not exist, then $u_s$ may not converge as $s \to 0^+$. The precise statement of our result is the following.

\begin{theorem}\label{thm-limsup-liminf}
  Let $\Omega:= B_1(0) \subset \R^N$, $N \ge 2$ be the unit ball. Then there exists $g \in L^{\infty}(\RR^N\setminus \Omega)$ with the property that the unique solutions $u_s$ of (\ref{main eq}) satisfy
  \begin{equation}
    \label{eq:limsup-liminf}
  \limsup_{s \to 0^+}u_s(x)=1, \quad \liminf_{s \to 0^+}u_s(x)=0 \qquad \qquad \text{for some $x \in \Omega$.}  
  \end{equation}
\end{theorem}

With the result above, we close the discussion of Question 1 and Question 2. In our next main result, we give an answer to Question 3 by showing that, under rather weak assumptions on $g$, the solution map $(0,1) \to L^\infty(\R^N)$, $s \mapsto u_s$ for the family of problems (\ref{main eq}) is of class $C^1$. Moreover, we derive an explicit formula for its derivative involving the logarithmic Laplacian. The result can be seen as an analogue to \cite[Theorem 1.1]{JSW20}, which is concerned with the corresponding family of $s$-dependent Poisson problems with homogeneous Dirichlet conditions.

To state the result, we first recall the definition of the {\em complementary Poisson kernel} of order $s \in (0,1)$ from \cite{JSW25}, which is given by 
  \begin{equation}
    \label{eq:complementary-poisson-def}
  P_s^{c}: \Omega \times \Omega \to \R, \qquad     P_s^{c}(x,z) = c_N \int_{\R^N\setminus \Omega}\frac{ P_s(z,y)}{|x-y|^{N}}\,dy.
  \end{equation}
We also need the fractional-logarithmic operator, very recently introduced by Chen, Chen, and Hauer in \cite{CH26}, given by 
  \begin{equation}
 \label{fractional-logarithmic-def}   
    [(-\Delta)^{s+\Log}u](x)= c_{N,s}\,\mathcal \PV\int_{\RR^N} \frac{u(x)-u(y)}{|x-y|^{N+2s}}\bigl(-2\ln|x-y|\bigr)\,dy+b_{N,s}(-\Delta)^s u(x)        
    \end{equation}
    with $b_{n,s} = \frac{d}{ds} c_{N,s}$. Note that, for $x \in \R^N$, $[(-\Delta)^{s+\Log}u](x)$ is well-defined for functions $u \in L^1_{s-\eps}(\R^N)$ which are of class $C^{2s+\eps}$ in a neighborhood of $x$ for some $\eps>0$. For functions $u \in C^2_c(\R^N)$, Chen, Chen, and Hauer in \cite{CH26} proved the pointwise identity
    \begin{equation}
      \label{eq:pointwise-identity-fractional-logarithmic}
      (-\Delta)^{s+\Log}u = L_\Delta (-\Delta)^s u = (-\Delta)^s L_\Delta u = \frac{d}{ds}[(-\Delta)^s u]. 
    \end{equation}
    In fact, this identity also holds if $u \in C^\alpha_c(\R^N)$ for some $\alpha>2s$. For the reader's convenience, we supply an elementary proof of this in Appendix~\ref{appD}. Finally, we fix
\begin{equation}
  \label{omega beta}
  \text{$U\subset \R^N$ as an open neighbourhood of $\Omega$, so that $\overline{\Omega}\subset U$.}
\end{equation}

\begin{theorem}\label{thm1.3}
  Let $N\geq 2$ and $\Omega \subset \RRN$ be an open bounded domain with $C^2$-boundary. Let $g \in L^{\infty}(\RR^N) \cap C^{\alpha}(U)$ for some $\alpha\in(0,1]$ where $U$ is given by \eqref{omega beta}, and let $u_s \in L^\infty(\R^N)$ be the unique solution of (\ref{main eq}) for $s \in (0,1)$. Then the map
$$
(0,\frac{\alpha}{2}) \to L^\infty(\Omega), \qquad s \mapsto u_s
$$
is of class $C^1$, and its derivative $\partial_s u_s \in L^{\infty}(\Omega)$ with respect to $s \in (0,\frac{\alpha}{2})$ is given by the unique solution of the fractional Dirichlet problem 
\begin{equation}
  \label{general-deriv-formula}
\left\{  \begin{aligned}
    (-\Delta)^s (\partial_s u_s)&=   \cPs^{\!\! c}\Bigl((-\Delta)^s g\big|_{\Omega}\Bigr)+ L_\Delta \bigl(1_{\Omega} (-\Delta)^s g\bigr) - (-\Delta)^{s+\mathrm{Log}}g  &&\qquad\text{in $\Omega$,}\\ 
    \partial_s u_s & =0&&\qquad\text{in $\RRN\setminus \Omega$.}
  \end{aligned}
\right.
\end{equation}
Moreover, if $(-\Delta)^s g \in L^1_0(\R^N)$ for some $s \in (0,\frac{\alpha}{2})$, then the right hand side of (\ref{general-deriv-formula}) can be written as
\begin{equation}
  \label{general-deriv-formula-special-case}
\cPs^{\!\! c}\Bigl((-\Delta)^s g\big|_{\Omega}\Bigr)- L_\Delta \bigl(1_{\R^N \setminus \Omega} (-\Delta)^s g\bigr)\qquad\text{in $\Omega$.}
\end{equation}
\end{theorem}

\begin{remark}
  (i) The assumption $g \in L^{\infty}(\RR^N) \cap C^{\alpha}(U)$ implies that $(-\Delta)^s$ is defined as a distribution on $\R^N$ and in pointwise sense as a $C^{\alpha-2s}$-function on $\Omega$ for $s \in (0,\frac{\alpha}{2})$. These definitions are consistent, see Lemma~\ref{observation-pointwise} below. In particular, $L_\Delta \bigl(1_{\Omega} (-\Delta)^s g\bigr)$ is well-defined as a function in $\Omega$.  Moreover, the extra assumption $(-\Delta)^s g \in L^1_0(\R^N)$ guarantees that also $L_\Delta \bigl(1_{\R^N \setminus \Omega} (-\Delta)^s g\bigr)$ is well-defined as a function in $\Omega$.

(ii) It is remarkable that, while the family of solutions $u_s$ of (\ref{main eq}) and its derivative $\partial_s u_s$ only depend on the values of $g$ in $\R^N \setminus \Omega$, it seems more convenient to express the derivative in terms of an extension of $g$ to all of $\R^N$, as it is required in the formulas (\ref{general-deriv-formula}) and (\ref{general-deriv-formula-special-case}).

\end{remark}

As a corollary of Theorem~\ref{thm1.3} and the positivity of the complementary Poisson kernel $  P_s^{\,c}$ we can derive a sufficient criterion for the pointwise monotonicity of the solution family $u_s$ in $s$.

\begin{corollary}
  \label{cor-thm-1-3}
  Let $N\geq 2$ and let $\Omega \subset \RRN$ be an open bounded domain with $C^2$-boundary. Let $g \in L^{\infty}(\RR^N) \cap C^{\alpha}(U)$ for some $\alpha\in(0,1],$ where $U$ is given by \eqref{omega beta}, and let $u_s \in L^\infty(\R^N)$ be the unique solution of (\ref{main eq}) for $s \in (0,1)$. Moreover, suppose that
  $$
  (-\Delta)^s g \in L^1_0(\R^N) \qquad \text{for all $s \in (0,\frac{\alpha}{2})$.}    
$$
  If $(-\Delta)^sg \geq 0\,\, (\leq 0)$ in $\R^N$ for $s \in (0,\frac{\alpha}{2})$, then the map
$$
(0,\frac{\alpha}{2}) \to \R, \qquad s \mapsto u_s(x)
$$
is increasing (decreasing) for every $x \in \Omega$. 
\end{corollary}

\begin{remark}
\label{example-monotonicity}
At first glance, the assumption on the sign of $(-\Delta)^s g$ may appear nontrivial, but there is a large class of functions $g$ for which $(-\Delta)^s g$ has a fixed sign for all $s \in (0,1)$. Indeed, let $N \ge 3$, let $f \in C_c^\beta(\R^N)$ be nonnegative and consider the Newtonian potential $g$ of $f$ given as the convolution $g=F_1\ast f \in L^\infty(\R^N) \cap C^{2+\beta}(\R^N)$. Here, for $s \in (0,1]$, we let 
$$
F_s:\RR^N\to[0,\infty], \quad F_s(z)=\kappa_{N,s}|z|^{2s-N}.
$$
denote the Riesz kernel, cf. \cref{sec.defis ops}. By the semigroup property of convolution with Riesz kernels, we then have
$$
g = F_s \ast f_s  \qquad \text{for all $s \in (0,1)$ with $f_s := F_{1-s} \ast f$}
$$
and thus $(-\Delta)^s g = f_s \ge 0$ in $\R^N$ for all $s \in (0,1)$. Hence, for this class of functions $g$, Corollary~\ref{cor-thm-1-3} applies and yields that
$$
(0,1) \to \R, \qquad s \mapsto u_s(x)
$$
is increasing for every $x \in \Omega$. We emphazise that this conclusion does not depend on the shape of $\Omega$.
\end{remark}

The paper is organized as follows. In \cref{priliminary}, we recall the necessary definitions of operators and function spaces used throughout the paper. In \cref{sec:flatness-bounded-s}, we prove Theorem \ref{thm1.0}. Theorem \ref{thm1.1} is proved throughout Section \ref{cont at s=0}, and Section \ref{diff at s=0}. Section \ref{sec:counterexample} is devoted to prove Theorem \ref{thm-limsup-liminf}. In Section \ref{diff for any s}, we prove Theorem \ref{thm1.3}. Finally, \cref{short proof of harnack}, \cref{weighted limit}, \cref{weighted limit2}, and \cref{appD} contain, respectively, a short proof of the classical nonlocal Harnack inequality, an asymptotic limit of the Poisson kernel, continuity of the Green operator, and some basic properties of fractional logarithmic Laplacian. 

\section{Preliminaries}\label{priliminary}
In this section, we collect some basic material that is needed for our purposes. We begin with the logarithmic Laplacian and the associated function spaces. Then, we recall some aspects of fractional operators, including fundamental solutions, Green and Poisson operators, and representation formulas.

\subsection{The logarithmic Laplacian and function spaces} The logarithmic Laplacian is a weakly singular Fourier integral operator associated with the symbol $2 \log |\cdot|.$ As mentioned earlier, the operator $L_{\Delta}$ appears in the study of $(-\Delta)^su$ in the limit $s \to 0$. Given a function $u \in C^{\alpha}_c(\RR^N)$ for some $\alpha>0,$ $L_{\Delta}u(x)$ can be uniquely defined by the asymptotic expansion.
\begin{align*}
    (-\Delta)^s u(x)=u(x)+sL_{\Delta}u(x)+o(s),
\end{align*}
where $o(s)$ denotes the small $o$ that satisfies $\frac{o(s)}{s} \to 0$ as $s\to 0.$ So, $L_{\Delta}u$ appears as the linear correction term in the small-order regime of the fractional Laplacian. The operator admits an explicit (weakly singular) integral representation $L_{\Delta}u(x)$, which, for $u\in C^{\alpha}_c(\RR^N)$, $\alpha>0$, and $x \in \RR^N$, is given by
\begin{align}\label{integral rep_log}
    L_{\Delta}u(x)=c_N \text{P.V.}\int_{B_1(x)}\frac{u(x)-u(y)}{|x-y|^N}\, dy-c_N\int_{B^c_1(x)}\frac{u(y)}{|x-y|^{N}}\, dy+\rho_N u(x).
\end{align}
Here, $B_r(x)$ is the ball centered at $x$ with radius $r>0$ and 
\begin{align*}
    c_N=\frac{\Gamma(\frac{N}{2})}{\pi^{\frac{N}{2}}}=\frac{2}{\omega_{N-1}}, \quad \rho_N=2 \log 2+\psi(\frac{N}{2})-\gamma,
\end{align*}
where as in the introduction $\omega_{N-1}$ denotes the $(N-1)$-dimensional volume of the unit sphere in $\RR^N$,  $\gamma=-\Gamma^{\prime}(1)$ is the Euler-Mascheroni constant, and $\psi=\frac{\Gamma^{\prime}}{\Gamma}$ is the Digamma function.

To extend the operator beyond smooth, compactly supported functions, one works with a weighted space $L^1_0(\RR^N)$ as in \eqref{L1_s space} with $s=0$ and the corresponding norm. Then, for a function $u \in L^{1}_0(\RR^N)$, the logarithmic Laplacian of $u$, $L_{\Delta}u$, is well-defined in the distributional sense as 
\begin{align*}
    (L_{\Delta}u)\phi =\int_{\RR^N}u (x) L_{\Delta}\phi(x)\, dx \quad \text{for}\quad \phi \in C^{\infty}_c(\RR^N).
\end{align*}
In the following, we use also the corresponding functional analytic framework for the logarithmic Laplacian. For this we define the space
\begin{align*}
  \mathbb{H}(\RR^N)=\left\{u\in L^2(\RR^N)\,:\, \iint_{\{x, y \in \RR^N\,:\,  |x-y|\leq 1\}}\frac{(u(x)-u(y))^2}{|x-y|^N}\, dx\, dy< \infty  \right\} 
\end{align*}
and, for $\Omega \subset \RR^N$ open and bounded, we consider
\begin{align*}
  \mathbb{H}_0(\Omega)=\left\{u \in \mathbb{H}(\RR^N)\,:\, u\equiv 0 \quad\text{on}\quad \RR^N\setminus \Omega\right\}.  
\end{align*}
The space $\mathbb{H}_0(\Omega)$ corresponds to the Dirichlet problem for $L_{\Delta}$ in $\Omega$ and is a Hilbert space with the inner product
\begin{align*}
    \langle u, v \rangle_{\mathbb{H}_0(\Omega)}=\frac{c_N}{2}\iint_{x, y \in \RR^N,  |x-y|\leq 1}\frac{(u(x)-u(y))(v(x)-v(y))}{|x-y|^N}\,dx\, dy.
\end{align*}
We denote the induced norm by $\|\cdot\|_{\mathbb{H}_0(\Omega)}$. We recall, that $\mathbb{H}_0(\Omega)$ is compactly embedded into $L^2(\Omega)$, where the space $L^2(\Omega)$ is identified as the space of functions in $L^{2}(\RR^N)$ with $u\equiv0$ on $\RR^N \setminus \Omega$, see \cite[Theorem 2.1]{CdP18}.\\

\noindent Furthermore, the bilinear form naturally associated with $L_{\Delta}$ is well-defined on $\mathbb{H}_0(\Omega)$ and given by
\begin{align*}
    \langle \cdot,\cdot \rangle_L&: \mathbb{H}_0(\Omega)\times \mathbb{H}_0(\Omega),\\
    &\qquad \qquad \langle u, v \rangle_L=\langle u,v \rangle-c_N \iint_{x, y \in \RR^N,  |x-y|\geq 1} \frac{u(x)v(y)}{|x-y|^N}\,dx\, dy+\rho_N\int_{\RR^N}u(x)v(x)\, dx.
\end{align*}
We note that the logarithmic Laplacian $L_{\Delta}u(x)$ is also well-defined by \eqref{integral rep_log} if $u\in L^1_0(\RR^N)$ is only Dini continuous at $x$, see \cite[Proposition 2.2]{CW19}. In this case, we also have the domain dependent representation
\begin{equation}\label{alt_integral rep_log}
L_{\Delta}u(x)=c_N \int_{\Omega}\frac{u(x)-u(y)}{|x-y|^N}\, dy-c_N\int_{\RR^N\setminus \Omega}\frac{u(y)}{|x-y|^{N}}\, dy+\left(h_{\Omega}(x)+\rho_N\right) u(x)    
\end{equation}
for $x\in \Omega$, where the function $h_{\Omega}$ is given by
\begin{equation}
  \label{def-h-Omega}
    h_{\Omega}: \Omega \to \RR,\qquad h_{\Omega}(x)=c_N \left(\int_{B_1(x)\setminus \Omega}\frac{dy}{|x-y|^N}-\int_{\Omega\setminus B_1(x)}\frac{dy}{|x-y|^N}\right).
\end{equation}

\subsection{Solution maps and related operators}\label{sec.defis ops}
To prepare for the analysis of boundary value problems, we now recall some basic tools from the theory of $(-\Delta)^s$ similar to \cite{JSW20}. Recall that, for $s \in (0,1)$, the fundamental solution of $(-\Delta)^s$ is given by $F_s:\RR^N\to\RR$, where\footnote{Recall, that $N\geq 2$ throughout our work}
\[
F_s(z)=\kappa_{N,s}|z|^{2s-N},\quad\kappa_{N,s}=\frac{\Gamma(\frac{N}{2}-s)}{4^s\pi^{N/2}\Gamma(s)}=\frac{s \Gamma(\frac{N}{2}-s)}{4^s\pi^{N/2}\Gamma(1+s)}.
\]
We note that 
\begin{equation}
  \label{eq:kappa-N-s-limit}
\frac{\kappa_{N,s}}{s} = \frac{\Gamma(\frac{N}{2}-s)}{4^s\pi^{N/2}\Gamma(1+s)}
\to c_N= \frac{\Gamma(\frac{N}{2})}{\pi^{N/2}}\qquad \text{as $s \to 0^+$.}
\end{equation}
The convolution with the fundamental solution is usually called the {\em Riesz operator} 
$$
f \mapsto \cFs_s f:= F_s * f.
$$
By the Hardy-Littlewood-Sobolev inequality (see \cite{LL01,L83}) this convolution defines a continuous linear map $\cFs_s: L^r(\RR^N) \to L^{p(r,s)}(\RR^N)$ for 
\begin{equation}\label{prs}
r \in (1,\frac{N}{2s}) \qquad \text{and}\qquad p(r,s):= \frac{rN}{N-2sr}.
\end{equation}
Given some exterior data $g$, the following holds for $s$-harmonic functions $u$ in $\Omega$, which equal $g$ outside, see \cite{nicola,BB99,BJK19}.
\begin{lemma}
\label{poisson-representation-new-zero}
Let $g\in L^1_s(\RR^N)$, which is bounded in a relative neighborhood of $\partial \Omega$. Then there exists a unique function $u: \RR^N \to \RR$ which is $s$-harmonic in $\Omega$, bounded in a neighborhood of $\partial \Omega$, and satisfies $u \equiv g$ in $\RR^N \setminus \Omega$.  
\end{lemma}
For $s \in (0,1)$, the Poisson kernel $P_s: \Omega \times (\RR^N\setminus \overline{\Omega}) \to \RR$ is defined as
\begin{align}\label{Psdef}
P_s(x,z):=c_{N,s}\int_\Omega \frac{G_s(x,y)}{|y-z|^{N+2s}}\, dy\qquad \text{for $x\in \Omega$, $z\in \RR^N\setminus\overline{\Omega}$.}
\end{align}
Here, $G_s: \RR^{2N}_* \to \RR^N$ denotes the Green kernel associated to $(-\Delta)^s$ in $\Omega$, with
$$
\RR^{2N}_* := \RR^N \times \RR^N \setminus \{(x,x)\:: \: x \in \RR^N\}.
$$
The Green kernel has the property that, for every $g\in H^1(\RR^N\setminus\Omega)\cap C^{1,\alpha}(\RR^N\setminus\Omega)$, the function $x \mapsto \int_{\RR^N\setminus\Omega}P_s(x,z)g(z)dz$
is the unique $s$-harmonic extension of $g$ on $\Omega$, which is bounded in a neighborhood of $\partial \Omega$, see e.g. \cite{BKK08}.\\

Next, we provide some preliminary convergence results, as $s \to 0^+$, related to homogeneous and nonhomogeneous fractional Dirichlet problems involving the operator $(-\Delta)^s$. We start by fixing some notation.\\

For $r \in (1,\frac{N}{2s})$ we define the Green operator
\begin{equation}
  \label{eq:G-op-definition}
\cGs_s \::\: L^{r}(\Omega) \to L^{p(r,s)}(\RR^N),\quad [\cGs_s f](x) = \int_{\Omega}G_s(x,y)f(y)\,dy \qquad \text{for $x \in \RR^N$,}
\end{equation}
where $p(r,s)$ is as in \eqref{prs}. We emphasize that $\cGs_s$ can also be defined on $L^r(\Omega)$ for any $r\in[1,\infty]$, but we do not need this in the following.
Analogously, for $s \in (0,1)$, we can define associated to the fractional Poisson kernel $P_s$ in \eqref{Psdef} the Poisson operator, cf. \eqref{intro-def-poisson},
\begin{align*}
\cPs \:&:\: L^{\infty}(\RR^N\setminus \Omega)\cap C^\alpha (\RR^N\setminus \Omega) \to C^\infty(\Omega)\cap L^\infty(\RR^N),\\
&\qquad [\cPs h](x) = \Bigl(\int_{\RR^N\setminus \Omega}P_s(x,y)h(y)\,dy\Bigr)  1_{\Omega}(x) + h(x)1_{\RR^N\setminus \Omega}(x) \qquad \text{for $x \in \RR^N$.}
\end{align*}

Finally, to relate the Poisson kernel, resp. the Poisson operator, with the fundamental solution, resp. the Riesz operator, we define a function $H_s: \RR^{2N}_* \to \RR$, associated with $\Omega$, as follows: For fixed $x\in \Omega$, $H_s(x,\cdot): \RR^N \to \RR$ denotes the unique solution in $L^\infty(\RR^N)$ of
\[
\left\{\begin{aligned}
(-\Delta)^s_yH_s(x,\cdot)&=0&&\text{in $\Omega$,}\\
H_s(x,\cdot)&=F_s(x-\,\cdot)&& \text{on $\RR^N\setminus \Omega$.}
\end{aligned}\right.
\]
Moreover, if $x \in \RR^N \setminus \Omega$, we set 
$H_s(x,\cdot): = F_s(x-\cdot)$ on $\RR^N \setminus \{x\}$.
By the maximum principle, we then have 
\begin{equation*}
  \label{eq:H-F-est}
0 \le H_s(x,y) \le F(x-y) \qquad \text{for $x,y \in \RR^{2N}_*$.}  
\end{equation*}
Consequently, for $r \in (1,\frac{N}{2s})$, we can define an operator 
$$
\cHs_s: L^{r}(\Omega) \to L^{p(r,s)}(\RR^N), \qquad [\cHs_s f](x) = \int_{\Omega}H_s(x,y)f(y)\,dy \qquad \text{for $x \in \RR^N$.}
$$

Recall that the general Poisson kernel representation formula for $s$-harmonic functions in $\Omega$ with prescribed values in $\RR^N \setminus \Omega$, see e.g. \cite[Theorem 1.2.3 and Lemma 3.1.2)]{nicola} (see also \cite{B99,BKK08}) states:

\begin{lemma}
\label{poisson-representation-new}
Under the assumptions of Lemma~\ref{poisson-representation-new-zero}, we have 
$$
u(x)= -\int_{\RR^N\setminus \Omega}g(z) (-\Delta)^s_z G_s(x,z)\,dz=\int_{\RR^N\setminus \Omega} P_s(x,z)g(z)\, dz \qquad \text{for $x \in \Omega$.}
$$
\end{lemma}

Hence, as an immediate consequence of Lemma~\ref{poisson-representation-new} we have the following formulas, which are stated in \cite[Lemma 3.2.vi)]{nicola}: In $\Omega$, it holds
$$
\int_{\RR^N\setminus \Omega}F_s(x-z) P_s(\cdot,z)\ dz
=-\int_{\RR^N\setminus \Omega}F_s(x-z) (-\Delta)^s_z G_s(\cdot ,z)\ dz=\left\{\begin{aligned} &F_s(x- \,\cdot\,) && \text{for $x \in \RR^N \setminus \overline{\Omega}$;}\\
&H_s(x,\cdot) &&\text{for $x \in \Omega$.}
\end{aligned}\right.
$$
Using the fact that $H_s(x,y)= F_s(x-y)$ if $x \in \RR^N \setminus \Omega$ or $y \in \RR^N \setminus \Omega$ and that these identities are then continuous in $x\in \RR^N\setminus\{y\}$ for fixed $y\in \Omega$, we have the compact form 
\begin{equation*}
  \label{eq:splitting-prelim-2}
  \begin{split}
H_s(x,y)&=-\int_{\RR^N\setminus \Omega}F_s(x-z)(-\Delta)^s_z G_s(y,z)\ dz\\
&=\int_{\RR^N\setminus \Omega}F_s(x-z)P_s(y,z)\ dz\qquad \text{for $x \in \RR^N$, $y \in \Omega$.}
\end{split}
\end{equation*}

Next, we recall some auxiliary lemmas from \cite{JSW20}. Consider the Dirichlet problem with zero exterior data, that is,
\begin{align*}
 \begin{cases}
    (-\Delta)^s u =f\,\,\, &\text{in}\,\,\, \Omega\\
    u=0\,\,\, &\text{in}\,\,\, \RRN\setminus\Omega,
\end{cases}   
\end{align*}
We note that, for sufficiently smooth $f$, one can formally transfer information between this formulation and the one considered in this paper, that is, the zero right-hand side with non-zero exterior data $g$, provided that $g$ possesses suitable regularities. Assuming the required regularities on $f$ (possibly up to the boundary of $\Omega$),  we have the following results.
 \begin{lemma}[Lemma 4.2, \cite{JSW20}]\label{Lem: lemma 4.1}We have the following estimates for $s \in (0, \frac{1}{2}].$\hspace{1em}
  \begin{enumerate}
  \item[(i)] If $f\in C^{\alpha}_c(\RR^N)$ for some $\alpha>0$, then  
    \begin{equation*}
      \label{eq:rn-case-first-claim}
\Big|\Bigl[\frac{\cFs_s-\textnormal{id}}{s}f\Bigr] (x)+[L_{\Delta} f](x)\Big|\leq s\ C \|f\|_{C^{\alpha}(\RR^N)}\quad\text{ for $x\in \RR^N,$}
    \end{equation*}
with some constant $C= C(N,\alpha,\supp f)>0$.
\item[(ii)] If $f\in C^{\alpha}(\overline{\Omega})$ for some $\alpha>0$, then 
  \begin{equation}
    \label{eq:rn-case-second-claim}
\Big|\Bigl[\frac{\cFs_s-\textnormal{id}}{s}E_{\Omega} f\Bigr] (x)+[L_{\Delta} E_{\Omega} f](x)\Big|\leq \frac{s\ C}{\varepsilon} \delta_{\Omega}(x)^{-\varepsilon}\|f\|_{C^{\alpha}(\overline{\Omega})}
\end{equation}
for $x\in \Omega$, $\varepsilon \in (0,1)$, with some constant $C=C(N,\Omega,\alpha)>0$.
  \end{enumerate}
\end{lemma} 

\begin{lemma}[Lemma 4.4, \cite{JSW20}]\label{Lem: lemma 4.2}
Let $f\in L^\infty(\Omega)$. Then 
\[
\limsup_{s \to 0^+} \frac{1}{s^2}\|\delta_{\Omega}^\varepsilon \cHs_s f\|_{L^\infty(\Omega)} \le C \|f\|_{L^\infty(\Omega)}\qquad \text{ for every $\varepsilon>0$}
\]
with a constant $C=C(N,\Omega,\varepsilon)$.
\end{lemma}


\subsection{A further preliminary estimate}
\label{sec:furth-prel-results}

In the following, we note the simple estimate
$$
|a^{-\alpha}-b^{-\alpha}|\le \alpha |a-b|(a^{-\alpha-1} + b^{-\alpha-1}) \qquad \text{for $a,b,\alpha>0$,}
$$
which implies that
\begin{equation}
  \label{alpha-est}
\Bigl||y|^{-\alpha}-|z|^{-\alpha} \Bigr|\le \alpha |y-z|\bigl(|y|^{-\alpha-1}-|z|^{-\alpha-1}\bigr) \qquad \text{for $y,z \in \R^N \setminus \{0\}$.}  
\end{equation}
As a consequence, for $0<r<R< \infty$ and $z \in B_{r}(0), y \in \RR^N \setminus B_R(0)$ we have 
\begin{align}\label{crR estimate}
\Bigl||y|^{-N}- |z-y|^{-N}\Bigr| &\le N|z| \Bigl(|y|^{-N-1}+|z-y|^{-N-1}\Bigr) \le  \frac{Nr}{R}|y| \Bigl(|y|^{-N-1}+(|y|-|z|)^{-N-1}\Bigr)\nonumber\\
                                   & \le \frac{Nr}{R}|y| \Bigl(|y|^{-N-1}+\bigl(|y|-\frac{r}{R}|y|\bigr)^{-N-1}\Bigr)
 = C_{r,R}|y|^{-N}
\end{align}
with $C_{r,R}:= \frac{Nr}{R}\Bigl( 1+ (1-\frac{r}{R})^{-N-1}\Bigr)$.

\section{Flatness of bounded \texorpdfstring{$s$}{s}-harmonic functions for small \texorpdfstring{$s$}{s}}
\label{sec:flatness-bounded-s}

The purpose of this section is to give the proof of Theorem~\ref{thm1.0}. We start with the following comparison estimate for $s$-harmonic functions in balls. 

\begin{lemma}\label{lem-thm1.0}
 Let $s\in(0,1)$, $0<r<1$, $R\geq 2,$ $x \in \RR^N$, and let $u \in L^\infty(\RR^N)$ be $s$-harmonic in $B_r(x)$. Then we have 
$$
|u(z)-u(x)| \le   m_{s,r,R} \|u\|_{L^\infty(\RR^N)} \qquad \text{for $z  \in B_{\frac{r}{2}}(x)$}
$$
with $m_{s,r,R}:= \Bigl(s2^{N+2}\Bigl(1+\ln(\frac{R}{2r})\Bigr) +\frac{ r^{2s}}{\Gamma(1-s)}\Bigl( \bigl(1-  \bigl(\frac{3}{4}\bigr)^s\bigr)   + C_{r,R}\Bigr)$.
\end{lemma}

\begin{proof}
By translation and relabeling we may assume $x = 0$ and fix $R \ge 2$. The Poisson kernel representation~(\ref{eq:poisson-rep-balls}) then yields, for $z\in B_r:=B_r(0)$,
$$
u(z)=\frac{\Gamma(\frac{N}{2})}{\pi^{\frac{N}{2}}\Gamma(1-s)\Gamma(s)}\int_{\RR^N\setminus B_r}\frac{(r^2-|z|^2)^su(y)}{(|y|^2-r^2)^s|z-y|^N}\, dy.
$$
Consequently, 
$$
|u(z)-u(0)| \le A_z + B_z
$$
with 
$$
A_z:= \frac{\Gamma(\frac{N}{2})}{\pi^{\frac{N}{2}}\Gamma(1-s)\Gamma(s)}
\int_{B_R\setminus B_r} \frac{u(y)}{(|y|^2-r^2)^s} \Bigl| \frac{(r^2-|z|^2)^s}{|z-y|^N}-\frac{r^{2s}}{|y|^N}\Bigr|\, dy,
$$
and 
$$
B_z:= \frac{\Gamma(\frac{N}{2})}{\pi^{\frac{N}{2}}\Gamma(1-s)\Gamma(s)}
\int_{\RR^N\setminus B_R} \frac{u(y)}{(|y|^2-r^2)^s} \Bigl| \frac{(r^2-|z|^2)^s}{|z-y|^N}-\frac{r^{2s}}{|y|^N}\Bigr|\, dy. 
$$
We note that, since $|y-z| \ge \frac{|y|}{2}$ for $z \in B_{\frac{r}{2}}$, $y \in \RR^N \setminus B_r$, we have  
\begin{align*}
A_z &\le \frac{\Gamma(\frac{N}{2})\|u\|_{L^\infty(B_R)}}{\pi^{\frac{N}{2}}\Gamma(1-s)\Gamma(s)} \int_{B_R \setminus B_r}  (|y|^2-r^2)^{-s}  \Bigl| \frac{(r^2-|z|^2)^s}{|z-y|^N}-\frac{r^{2s}}{|y|^N}\Bigr|\, dy\\
&\le \frac{\Gamma(\frac{N}{2})\|u\|_{L^\infty(\RR^N)}}{\pi^{\frac{N}{2}}\Gamma(1-s)\Gamma(s)} \int_{B_R \setminus B_r} (|y|^2-r^2)^{-s}  \Bigl( \frac{r^{2s}}{|z-y|^N} +\frac{r^{2s}}{|y|^N}\Bigr)\, dy\\
&\le \frac{(2^N+1) \Gamma(\frac{N}{2})\|u\|_{L^\infty(\RR^N)}r^{2s}}{\pi^{\frac{N}{2}}\Gamma(1-s)\Gamma(s)} \int_{B_R \setminus B_r} (|y|^2-r^2)^{-s}  |y|^{-N}\, dy
\end{align*}
where, since $\frac{R}{r}>2$,
\begin{align*}
 \int_{B_R \setminus B_r} (|y|^2-r^2)^{-s}|y|^{-N}\, dy&=\omega_{N-1}\int_r^R(\tau^2-r)^{-s}\tau^{-1}\, d\tau=\omega_{N-1}r^{-2s}\int_{1}^{R/r}(t^2-1)^{-s}
t^{-1}\, dt\\
&\leq \omega_{N-1}r^{-2s}\Bigl(\int_{1}^{2}(t-1)^{-s}(t+1)^{-s}
\, dt+\int_2^{R/r}t^{-1-2s}\, dt\Bigr)\\
&\leq \omega_{N-1}r^{-2s}\Bigl(\frac{1}{1-s}+\frac{1}{2s}\Big(2^{-2s}-\Big(\frac{R}{r}\Big)^{-2s}\Bigr).
\end{align*}
Since $(0,1)\ni s\mapsto \frac{1}{2s}\Big(2^{-2s}-\Big(\frac{R}{r}\Big)^{-2s}\Bigr)$ is decreasing, we have
$$
\frac{1}{2s}\Big(2^{-2s}-\Big(\frac{R}{r}\Big)^{-2s}\Bigr)\leq \lim_{s\to 0^+} \frac{1}{2s}\Big(2^{-2s}-\Big(\frac{R}{r}\Big)^{-2s}\Bigr)=\ln(\frac{R}{2r})
$$
and therefore we obtain
\begin{align*}
A_z &\le \frac{(2^N+1) \Gamma(\frac{N}{2})\|u\|_{L^\infty(\RR^N)}\omega_{N-1}}{\pi^{\frac{N}{2}}\Gamma(1-s)\Gamma(s)}\Bigl(\frac{1}{1-s}+\ln(\frac{R}{2r})\Bigr)\\
&\le s \frac{2(2^{N} +1)\|u\|_{L^\infty(\RR^N)}}{ \Gamma(2-s)\Gamma(1+s)}\Bigl(1+\ln(\frac{R}{2r})\Bigr)\\
&\le s  2^{N+2}\|u\|_{L^\infty(\RR^N)} \Bigl(1+\ln(\frac{R}{2r})\Bigr).
\end{align*}
Moreover, since $|y|^2-r^2 \ge (1-\frac{r}{R})|y|^2$ for $y \in \RR^N \setminus B_R$, we estimate
\begin{align*}
  B_z &\le \frac{\Gamma(\frac{N}{2})\|u\|_{L^\infty(\RR^N)}}{\pi^{\frac{N}{2}}\Gamma(1-s)\Gamma(s)}
 \int_{\RR^N \setminus B_R}(|y|^2-r^2)^{-s}  \Bigl| \frac{(r^2-|z|^2)^s}{|z-y|^N}-\frac{r^{2s}}{|y|^N}\Bigr|   \, dy
\\
      &\le \frac{\Gamma(\frac{N}{2})\|u\|_{L^\infty(\RR^N)}(1-\frac{r}{R})^{-s}}{\pi^{\frac{N}{2}}\Gamma(1-s)\Gamma(s)} \int_{\RR^N \setminus B_R}|y|^{-2s}  \Bigl| \frac{(r^2-|z|^2)^s}{|z-y|^N}-\frac{r^{2s}}{|y|^N}\Bigr| \, dy\\  
      &\le \frac{\Gamma(\frac{N}{2})\|u\|_{L^\infty(\RR^N)}(1-\frac{r}{R})^{-s}}{\pi^{\frac{N}{2}}\Gamma(1-s)\Gamma(s)} \int_{\RR^N \setminus B_R}|y|^{-2s} \Bigl(  \frac{|(r^2-|z|^2)^s-r^{2s}|}{|y|^N} + (r^2-|z|^2)^s \Bigl| |z-y|^{-N}-|y|^{-N}\Bigr| \Bigr)\, dy\\  
      &\le \frac{\Gamma(\frac{N}{2})\|u\|_{L^\infty(\RR^N)}(1-\frac{r}{R})^{-s}}{\pi^{\frac{N}{2}}\Gamma(1-s)\Gamma(s)} \times\\
     & \left[ \bigl(r^{2s} - (r^2-|z|^2)^s\bigr)
       \int_{\RR^N \setminus B_R}|y|^{-(N+2s)} \,dy + r^{2s} \int_{\RR^N \setminus B_R} |y|^{-2s} \Bigl| |z-y|^{-N}-|y|^{-N}\Bigr| \Bigr)\, dy\right]\\
        &\overset{\eqref{crR estimate}}{\leq} \frac{\Gamma(\frac{N}{2})\|u\|_{L^\infty(\RR^N)}(1-\frac{r}{R})^{-s}}{\pi^{\frac{N}{2}}\Gamma(1-s)\Gamma(s)}  \Bigl( \bigl(r^{2s} - (r^2-|z|^2)^s\bigr)
         + r^{2s} C_{r,R} \Bigr)\int_{\RR^N \setminus B_R} |y|^{-(N+2s)} \, dy\\  
        &\le \frac{ R^{-2s}}{s} \frac{ \|u\|_{L^\infty(\RR^N)}(1-\frac{r}{R})^{-s}}{ \Gamma(1-s)\Gamma(s)} \Bigl( \bigl(r^{2s} - (r^2-|z|^2)^s\bigr)
         + r^{2s} C_{r,R}\Bigr)\\  
        &\le \frac{ R^{-s} \|u\|_{L^\infty(\RR^N)}(R-r)^{-s}}{s \Gamma(1-s)\Gamma(s)} \Bigl( \bigl(r^{2s} -  \bigl(\frac{3r^2}{4}\bigr)^s\bigr)
         + r^{2s} C_{r,R}\Bigr)\\  
    &\le \frac{ r^{2s}R^{-s} \|u\|_{L^\infty(\RR^N)}(R-r)^{-s}}{s \Gamma(1-s)\Gamma(s)} \Bigl( \bigl(1-  \bigl(\frac{3}{4}\bigr)^s\bigr)
         + C_{r,R}\Bigr)\\ 
    &\le \frac{ r^{2s} \|u\|_{L^\infty(\RR^N)}}{\Gamma(1-s)} \Bigl( \bigl(1-  \bigl(\frac{3}{4}\bigr)^s\bigr)
         + C_{r,R}\Bigr), 
\end{align*}
where we have used in the last step that $R \ge 2$ and $R-r\ge 1$ by assumption. Combining the above estimates of $A_z$ and $B_z,$ we finish the proof.
\end{proof}

\begin{corollary}
  \label{cor-thm1.0}
Let $s_0,r \in(0,1)$, $x \in \RR^N$, and let $u_s \in L^\infty(\RR^N)$ be $s$-harmonic in $B_r(x)$ for $s \in (0,s_0]$. Then for every $\varepsilon >0$ there exists $s_\varepsilon \in (0,s_0)$, independent of $x$ and $s_0$, with 
$$
|u(z)-u(x)| \le \varepsilon \qquad \text{for $z  \in B_{\frac{r}{2}}(x)$, $s \in (0,s_\varepsilon)$.}
$$
\end{corollary}

\begin{proof}
Let $\eps>0$. The explicit representation of $C_{r,R}$ in \eqref{crR estimate} allows to choose $R>2$ large so that $C_{r, R}\leq \frac{\varepsilon}{2+2\|u\|_{L^{\infty}(\RR^N)}}$. Hence, Lemma~\ref{lem-thm1.0} implies, for $z  \in B_{\frac{r}{2}}(x)$,
\begin{align*}
|u(z)-u(x)|\leq \|u\|_{L^{\infty}(\RR^N)}\Bigl(s2^{N+2}\Bigl(1+\ln(\frac{R}{2r})\Bigr) +\frac{ r^{2s}}{\Gamma(1-s)}\Bigl( \bigl(1-  \bigl(\frac{3}{4}\bigr)^s\bigr)\Bigr)+  \frac{\varepsilon}{2},
\end{align*} 
from which it follows that we can fix $s_{\eps}\in(0,s_0]$ as claimed.

\end{proof}

\begin{proof}[Proof of Theorem~\ref{thm1.0}]
This follows in a standard way from Corollary~\ref{cor-thm1.0} by an argument using a chain of balls connecting arbitrary two points in $\Omega$. 
\end{proof}

\section{Continuity of the solution at \texorpdfstring{$s=0$}{s =0}}\label{cont at s=0}
In this section, we prove Part (i) of \cref{thm1.1}. We state the following useful lemma from \cite{JSW20}.
\begin{lemma}[Lemma 2.3, \cite{JSW20}]\label{int:la}
	Let $N\geq 2$, $\Omega\subset \RR^N$ be an open bounded set with $C^{2}$-boundary, $a\in(-1,1)$, and $\lambda \in [-N,1)$. Moreover, for $\eta>0$, we define 
$$
m(a,\lambda,\eta):=
\frac{1}{1-a}  \left\{
  \begin{aligned}
&\frac{1}{1-\lambda} \min \Bigl\{\frac{1}{\lambda-a},1+|\ln \eta|\Bigr\} &&\qquad \text{if $\lambda>a$,}\\
&1\:+\: \eta^{\lambda-a} \min \Bigl\{\frac{1}{a-\lambda},\frac{1}{1-\lambda}(1+|\ln \eta|)\Bigr\} &&\qquad \text{if $\lambda<a$,}\\
&\frac{1}{1-\lambda}\Bigl(1+|\ln \eta|\Bigr) && \qquad \text{if $\lambda=a$.}
  \end{aligned}
\right.
$$
Then there exists a constant $C=C(N,\Omega)>0$ such that
\begin{align*}
\int_{\Omega}\delta_{\Omega}(y)^{-a}|x-y|^{\lambda-N}\ dy\;\leq\; C\frac{m(a,\lambda,\delta_{\Omega}(x))}{1+\delta_{\Omega}(x)^{N-\lambda}} \qquad \text{for $x \in \RR^N \setminus \overline \Omega\:$}
\end{align*}
and
\begin{align}\label{eq:(ii)-int-est}
\int_{B_{r_\Omega}(x) \setminus \Omega}\delta_{\Omega}(y)^{-a}|x-y|^{\lambda-N}\ dy
\;\leq\; C\: m(a,\lambda,\delta_{\Omega}(x)) \qquad \text{for $x \in \Omega$,}
\end{align}
where $r_\Omega:= \diam \:\Omega+1$. 
\end{lemma}

The following Theorem can be seen as a generalization of Theorem~\ref{thm1.1}\,(i).

\begin{theorem}\label{thm1.1-section}
  Let $N\geq 2$, $\Omega \subset \RRN$ be an open bounded domain with $C^2$-boundary, and let
  $$
  g \in L^{\infty}_{\loc}(\RR^N\setminus \Omega)\,
  \cap \bigcap \limits_{s \in (0,1)}L^1_s(\RR^N).
  $$
  Moreover, let
  $$
  \tilde g(r):= \fint_{S_r}g(y)\,d\sigma(y),
  $$
  which is well-defined for a.\,e. $r>0$. If 
  \begin{equation}
   \label{assumption-g-tilde} 
  \int_{M_R^\delta} |\tilde g(r)|dr = o(R) \quad \text{as $R \to \infty$}\qquad \text{for every $\delta>0$,} 
  \end{equation}
 where $M_R^\delta := \{r \in [R,2R]\::\: |\tilde g(r)|\ge \delta\}$, then  
 $u_s \to 0$ almost uniformly in $\Omega$ as $s \to 0^+.$
\end{theorem}

\begin{proof} For any $x\in \Omega,$ the solution to \eqref{main eq} can be explicitly written as
\begin{align*}
u_s(x)=\int_{\RR^N\setminus \Omega}P_s(x, y)g(y)\, dy=\int_{B_{r_{\Omega}}(x)\setminus \Omega}P_s(x, y)g(y)\,dy+\int_{B^c_{r_{\Omega}}(x)}P_s(x, y)g(y)\,dy=\mathrm I(x)+\mathrm {II} (x)
\end{align*}
where $r_{\Omega}:=\text{diam}\,\Omega+1.$ We estimate each term below.

\vspace{.5cm}
\noindent{\bf Estimate of $\mathrm{I}(x)$:} Denote $K=\{x\in \RR^N\setminus \Omega\;:\; \delta_{\Omega}(x)<r_{\Omega}\}.$ We use \cite[Theorem 2.10]{Chen99} and \eqref{eq:(ii)-int-est} of Lemma \ref{int:la} in the following estimate for $x \in \Omega$:
\begin{align*}
\mathrm{I}(x)=\int_{B_{r_{\Omega}}(x)\setminus \Omega}P_s(x, y)g(y)\,dy &\leq C s(1-s)\delta_{\Omega}(x)^s \int_{B_{r_{\Omega}}(x)\setminus \Omega} \frac{g(y)}{\delta_{\Omega}(y)^s(1+r^{-1}_0\delta_{\Omega}(y))^s|x-y|^N}\,dy\\
&\leq C s(1-s)\delta^s_{\Omega}(x)\|g\|_{L^{\infty}(K)}\int_{B_{r_{\Omega}}(x)\setminus \Omega} \delta_{\Omega}(y)^{-s}|x-y|^{-N}\,dy\\
&\overset{\text{Lemma \ref{int:la}}}{\leq} C s \delta^s_{\Omega}(x)\|g\|_{L^{\infty}(K)}\left(1+\delta^{-s}_{\Omega}(x)\min\left\{\frac{1}{s}, 1+|\log\delta_{\Omega}(x)|\right\}\right)\\
&\leq s C(\Omega, \|g\|_{L^{\infty}(K)}) \left(1+|\log \delta_{\Omega}(x)|\right)=sC_x. 
\end{align*}
Moreover, the RHS goes to zero as $s \to 0^+$ almost uniformly in $\Omega$. Next, we estimate the second integral.

\vspace{.5cm}
\noindent{\bf Estimate of $\mathrm{II}(x)$:} To estimate $\mathrm{II}(x),$ we use the assumptions on $g$ and \cref{a limit}, which yields the limit
\begin{equation}\label{expand ps fs}
    \lim_{|y|\to \infty}\frac{|y|^{2s+N}P_s(x, y)}{c_{N, s}}= \int_{\Omega}G_s(x, y)\, dy=:\mathcal{K}_s(x) \geq 0 \qquad \text{uniformly in $s \in (0,1)$ and $x \in \Omega$.}
\end{equation}
Note that $\mathcal{K}_s$ is the unique solution of the fractional torsion problem
$$
(-\Delta)^s \mathcal{K}_s =1 \quad \text{in $\Omega$,}\qquad
\mathcal{K}_s  = 0\quad \text{in $\RR^N \setminus \Omega$,}
$$
and it satisfies $\mathcal{K}_s(x) \to \mathcal{K}_0(x) =1_{\Omega}>0$ as $s\to 0^+$ almost uniformly in $\Omega$, see for instance \cite[Remark 3.8]{JSW20}.
The asymptotics \eqref{expand ps fs} allow to write the fractional Poisson kernel as
\begin{align*}
    P_s(x, y)= s\cdot\frac{c_{N,s}}{s}\mathcal{K}_s(x)|y|^{-(2s+N)}+s\, o\bigl(|y|^{-(2s+N)}\bigr)
\end{align*}
where $o(|y|^{-(2s+N)}) \to 0$ as $|y|\to \infty$ uniformly in $s\in (0, 1)$. 
Employing this in $\mathrm{II}(x)$, we get
\begin{align}\label{second estimate}
\mathrm{II}(x)&=\int_{B^c_{r_{\Omega}}(x)}P_s(x, y)g(y)\, dy\nonumber\\
&=c_{N, s}\mathcal{K}_s(x) \int_{B^c_{r_{\Omega}}(x)} |y|^{-(2s+N)}g(y)\, dy+ s \int_{B^c_{r_{\Omega}}(x)}o\bigl(|y|^{-(2s+N)}\bigr)g(y)\, dy.
\end{align}
Now using the assumption of $g,$ we estimate the first integral, which then also gives an estimate for the second integral. For $k \in \mathbb{N},r>0$ and $\delta>0$, we denote
\begin{align*}
M^{\delta}_{r}(k):=\left\{r\in [2^k r, 2^{k+1}r]\,:\,|\tilde{g}(r)|\geq \delta\right\},\\
N^{\delta}_{r}(k):=\left\{r\in [2^k r, 2^{k+1}r]\,:\,|\tilde{g}(r)|<\delta\right\}.
\end{align*}
Moreover, we split the integral in the following way.
\begin{align*}
&\Bigl|\int_{B^c_{r_{\Omega}}(x)}|y|^{-(2s+N)}g(y)\, dy\Bigr| \le \sum_{k=0}^{\infty}\int_{2^k r_{\Omega}}^{2^{k+1}r_{\Omega}} r^{N-1}\Bigl|\fint_{S_r}r^{-(2s+N)}g(y)\, d\sigma(y)\Bigr| dr\\
&= \sum_{k=0}^{\infty}\int_{2^k r_{\Omega}}^{2^{k+1}r_{\Omega}}  r^{-(1+2s)}|\tilde{g}(r)|\, dr\\
&\le \sum_{k=0}^{\infty}\Bigl(\int_{M^{\delta}_{r_{\Omega}}(k)} r^{-(1+2s)}|\tilde{g}(r)|\, dr+ \delta \int_{N^{\delta}_{r_{\Omega}}(k)} r^{-(1+2s)}\, dr\Bigr),
\end{align*}
where 
\begin{equation*}
  \sum_{k=0}^{\infty}\int_{N^{\delta}_{r_{\Omega}}(k)}r^{-(1+2s)}\, dr \le \sum_{k=0}^{\infty} \int_{2^k r_\Omega}^{2^{k+1} r_\Omega}r^{-(1+2s)}\,dr = \int_{1}^\infty r^{-(1+2s)}dr = \frac{r^{-2s}_{\Omega}}{2s},
\end{equation*}
and, for every $k_0 \ge 1$,
\begin{align*}
 & \sum_{k=0}^{\infty} \int_{M^{\delta}_{r_{\Omega}}(k)} r^{-(1+2s)}|\tilde{g}(r)|\, dr\leq 
                                                                                     \sum_{k=0}^{k_0-1} \int_{M^{\delta}_{r_{\Omega}}(k)} r^{-(1+2s)}|\tilde{g}(r)|\, dr + \sum_{k=k_0}^{\infty} \int_{M^{\delta}_{r_{\Omega}}(k)} r^{-(1+2s)}|\tilde{g}(r)|\, dr\\
  &\le C_1(k_0,g) + C(k_0)\sum_{k=k_0}^\infty (2^{k} r_\Omega)^{-2s}= C_1(k_0,g)  + \frac{C_2(k_0) 2^{-2s k_0} r_\Omega^{-2s}}{1-2^{-2s}}\\
&\le C_1(k_0,g) + \frac{C_2(k_0)}{1-2^{-2s}}  
\end{align*}
with
$$
C_1(k_0,g) := \int_{1}^{2^{k_0} r_\Omega}|\tilde{g}(r)|\, dr \quad \text{and}\quad C_2(k_0) = \Bigl(\sup_{k \ge k_0} \frac{1}{2^{k}r_\Omega}\int_{M^{\delta}_{r_{\Omega}}(k)}|\tilde{g}(r)|\, dr\Bigr).
$$
Consequently, we have
\begin{align*}
c_{N, s}\mathcal{K}_s(x) \Bigl|\int_{B^c_{r_{\Omega}}(x)} |y|^{-(2s+N)}g(y)\, dy\Bigr| \le c_{N, s}\mathcal{K}_s(x)\Bigl(\delta \frac{r^{-2s}_{\Omega}}{2s} + C_1(k_0,g) + \frac{C_2(k_0)}{1-2^{-2s}}  \Bigr)
\end{align*}
where $\frac{c_{N, s}}{s} \to c_N$ (\text{see the explicit form of $c_{N, s}$ in \eqref{defn_c(N, s)}}),  $\frac{s}{1-2^{-2s}} \to \frac{1}{\log 2}$, and $\mathcal{K}_s(x) \to 1$ almost uniformly in $\Omega$. Consequently
$$
\limsup_{s \to 0^+}\Bigl( c_{N, s}\mathcal{K}_s(x) \Bigl|\int_{B^c_{r_{\Omega}}(x)} |y|^{-(2s+N)}g(y)\, dy\Bigr|\Bigr) \le \delta \frac{c_N}{2} + \frac{C_2(k_0)}{\log 2}
$$
almost uniformly in $x \in \Omega$. Since this holds for every $\delta>0$, $k_0 \ge 1$ and $C_2(k_0) \to 0$ as $k_0 \to \infty$ by assumption \eqref{assumption-g-tilde}, we have
$$
c_{N, s}\mathcal{K}_s(x) \Bigl|\int_{B^c_{r_{\Omega}}(x)} |y|^{-(2s+N)}g(y)\, dy\Bigr| \to 0 \qquad \text{as $s \to 0^+$}
$$
almost uniformly in $x \in \Omega$. As mentioned above, the same estimate also gives that the second integral of \eqref{second estimate} converges to $0$ uniformly in $x \in \Omega$ as $s \to 0^+$. Consequently, $\mathrm{II}(x) \to 0$ as $s \to 0^+$ almost uniformly in $x \in \Omega$. Since the same holds for $\mathrm{I}(x)$ by the previous estimate, we conclude that $u_s \to 0$ almost uniformly in $\Omega$ as $s\to 0^+,$ as claimed.
\end{proof}

We may now complete the

\begin{proof}[Proof of Theorem~\ref{thm1.1}(i)]
  If $c_g = 0$, the assumption $\tilde g(r) \to 0$ as $r \to \infty$ clearly implies assumption~(\ref{assumption-g-tilde}), and thus the claim follows from Theorem~\ref{thm1.1-section}. If $c_g \not = 0$, then we consider the functions $g_c:= g -c_g$ and $u_{s,c_g} = u_s -c_g$ in place of $g$ and $u_s$ for $s>0$, which then clearly satisfy the assumptions of Theorem~\ref{thm1.1-section} again. So it follows that $u_{s,c_g} \to 0$ as $s \to 0^+$ almost uniformly in $x \in \Omega$, and hence $u_s \to c_g$ as $s \to 0^+$ almost uniformly in $\Omega$, as claimed.
 \end{proof} 

\section{First order expansion}\label{diff at s=0}
In this section, we prove Part (ii) of Theorem~\ref{thm1.1} and Corollary~\ref{cor-harnack}. Before we proceed with the proof, we first state the following lemma, which provides the existence of the limit in (\ref{eq:def-L-g-pv-sense}). To prove the existence of this limit, we may clearly assume $c_g =0$, since otherwise we may replace $g$ by $g-c_g$.  We denote 
\[
\tilde r_\Omega:= \sup_{x \in \Omega}|x|+r_{\Omega},
\]
with $r_{\Omega}=\diam(\Omega)+1$ as before in Lemma \ref{int:la}.
\begin{lemma}
  \label{regarding-def-L-g-pv-sense}
  Let $g \in L^\infty(\R^N \setminus \Omega)$ satisfy~(\ref{eq:first-order-exp-assumption}) with $c_g=0$. Then 
  \begin{equation}
    \label{continuous-function-g}
 x \mapsto \delta_\Omega^\eps(x) \int_{B_r \setminus \Omega}\frac{g(y)}{|x-y|^N}\,dy   
  \end{equation}
  defines a continuous function on $\overline{\Omega}$ for every $\eps>0$, $r>\tilde r_\Omega$
and the limit 
    \begin{equation}
    \label{eq:remainder-g-est}
 \lim_{r \to \infty} \delta_\Omega^\eps(x) \int_{B_r \setminus \Omega}\frac{g(y)}{|x-y|^N}\,dy \qquad \text{exists uniformly in $x \in \Omega$.}
  \end{equation}
  \end{lemma}

  \begin{proof}
    Since for every $y \in \R^N \setminus \Omega$ and every $x \in \Omega$ we have
    $$
    \frac{\delta_\Omega^\eps(x)}{|x-y|^N} \le |x-y|^{\eps-N},
    $$
a standard estimate shows that (\ref{continuous-function-g}) defines a continuous function on $\overline{\Omega}$ for every $r > \tilde r_\Omega$. To prove that the limit in (\ref{eq:remainder-g-est}) exists uniformly in $x \in \Omega$, it thus  suffices to show that
    \begin{equation}
      \label{eq:sup-cauchy}
      \sup_{r>r'} \delta^{\varepsilon}_{\Omega}(x)\int_{B_r \setminus B_{r'}}\frac{g(y)}{|x-y|^{N}}dy \to 0 \qquad \text{as $r' \to \infty$ uniformly in $x \in \Omega$.}   
    \end{equation}
 For this we let $r>r' > \tilde r_\Omega$,  and we note that
    \begin{align*}
        \int_{B_r \setminus B_{r'}}\frac{g(y)}{|x-y|^{N}}dy &= \int_{r'}^{r} \int_{S_t}\frac{g(y)}{|x-y|^{N}}d\sigma(y)dt\\
                                                                  &= \omega_{N-1}\int_{r'}^{r}\frac{\tilde g(t)}{t} dt + \int_{B_r \setminus B_{r'}} g(y)\Bigl(|x-y|^{-N}-|y|^{-N}\Bigr)dy.
      \end{align*}
By (\ref{alpha-est}), we have the estimate
      $$
      \Bigl||x-y|^{-N}-|y|^{-N}\Bigr| \le N |x|\Bigl(|y|^{-N-1} + |x-y|^{-N-1}\Bigr) \le N2^{N+1} |x| |y|^{-N-1}\qquad \text{for $x \in \Omega$ and $|y| \ge 2\tilde r_\Omega$.}
      $$
      Therefore,   
      \begin{align*}
        \Bigl|\int_{B_r \setminus B_{r'}}\frac{g(y)}{|x-y|^{N}}dy\Bigr|&=\Bigl| \omega_{N-1}\int_{r'}^{r}\frac{\tilde g(r)}{r} dr\Bigr|+  \Bigl|\int_{B_r \setminus B_{r'}} g(y)\Bigl(|x-y|^{-N}-|y|^{-N}\Bigr)dy\, dr\Bigr|\\
        &\le \omega_{N-1}\int_{r'}^{r}\frac{|\tilde g(r)|}{r} dr  + N|x| \int_{B_r \setminus B_{r'}} \frac{|g(y)|}{|y|^{N+1}}dy\\
        &\le \omega_{N-1}\int_{r'}^{\infty}\frac{|\tilde g(r)|}{r} dr + N|x| \int_{\RR^N \setminus B_{r'}} \frac{|g(y)|}{|y|^{N+1}}dy,
      \end{align*}
      where
      $$
      \int_{r'}^{\infty}\frac{|\tilde g(r)|}{r} dr \to 0 \quad \text{as}\quad r^{\prime}\to \infty
      $$
      by (\ref{eq:first-order-exp-assumption}) with $c_g =0$ and
      $$
      \int_{\RR^N \setminus B_{r'}} \frac{|g(y)|}{|y|^{N+1}}dy \le \frac{\omega_{N-1}}{r'} \|g\|_{L^\infty(\RR^N\setminus \Omega)} \to 0 \qquad \text{as $r' \to \infty$.}
      $$
      Hence, (\ref{eq:sup-cauchy}) holds and the proof is finished.
  \end{proof}
  
In the following, let $N_\beta:=   \{z \in \R^N\::\: \delta_{\Omega}(z) < \beta\}$ for $\beta>0$ is the $\beta$-neighborhood of $\partial \Omega$ and we denote, $\Omega_{\beta}=\overline{\Omega}\cup N_{\beta}$ as a $\beta$-neighborhood of $\Omega$.

\begin{lemma}
  \label{local-Poisson-kernel-convergence}
  We have
  \begin{align}\label{first claim}
    \lim_{s \to 0^+} \frac{P_s(x, y)}{s}= \frac{c_N}{|x-y|^N} \qquad \text{for $x \in \Omega$, $y \in \R^N \setminus \Omega$.}
  \end{align}
Moreover, if $r>\tilde r_\Omega$ and $\beta>0$, then the convergence in (\ref{first claim}) is uniform in $x \in \Omega \setminus N_\beta$, $y \in B_r \setminus \Omega_{\beta}$.
\end{lemma}

\begin{proof}
Let $r>\tilde r_\Omega$ and $\beta>0$. It then suffices to show the uniform convergence in (\ref{first claim}) for $x \in \Omega \setminus N_\beta$, $y \in B_r \setminus\Omega_{\beta}$. Let us write
\begin{align*}
    \frac{P_s(x, y)}{s}= \int_{\Omega}\frac{G_s(x, z)}{|y-z|^{N+2s}}\,dz&=\int_{\Omega}\frac{F_s(x-z)}{s}\frac{C_{N,s}}{|y-z|^{N+2s}}\,dz-\frac{c_{N, s}}{s}\int_{\Omega}\frac{H_s(x, z)}{|y-z|^{N+2s}}\,dz\\
    &=\mathrm I (x, y)+\mathrm {II}(x, y).
\end{align*}

\vspace{.5cm}
\noindent{\bf Estimate of $I(x,y)$:} By definition, for $y \in \RR^N\setminus \Omega,$ $I$ can be written as
\begin{align*}
    \mathrm{I}(x,y)=\int_{\Omega}\frac{F_s(x-z)}{s}\frac{C_{N,s}}{|y-z|^{N+2s}}\,dz=\frac{1}{s}\left[\mathbb{F}_s \left(1_{\Omega}\frac{c_{N, s}}{|\cdot-y|^{N+2s}}\right)\right](x)=\frac{1}{s}\left[\mathbb{F}_s \Phi_{s, y}\right](x),
\end{align*}
where 
\begin{align*}
    \Phi_{s, y}(\cdot):=1_{\Omega}\frac{c_{N, s}}{|\cdot-y|^{N+2s}}.
\end{align*}
We note that for fixed $y \in \RR^N\setminus \Omega,$ $\Phi_{s, y}(\cdot)$ is a smooth function on $\Omega$. Moreover, we identify $\mathbb{F}_0$ with the identity operator. Hence, we have
\begin{align*}
    \mathrm{I}(x,y)=\left[\frac{\mathbb{F}_s-id}{s} \Phi_{s, y}\right](x)+\frac{\Phi_{s, y}(x)}{s}.
\end{align*}
We also note that
$$
L_{\Delta}\Phi_{s, y} = L_\Delta^\Omega \Phi_{s,y} + \tilde h_\Omega \Phi_{s,y}
$$
with
$$
[L_\Delta^\Omega \Phi_{s,y}](x) = \int_{\Omega}\frac{\Phi_{s,y}(x)-\Phi_{s,y}(z)}{|x-z|^N}\,dz
= c_{N,s}\int_{\Omega}\frac{|x -y|^{-N-2s}-|z-y|^{-N-2s}}{|x-z|^N}\,dz
$$
and $\tilde h_\Omega = h_{\Omega}+\uprho_N$ with $h_\Omega$ defined in (\ref{def-h-Omega}). Therefore,
\begin{align}
  &\left|\mathrm{I}(x,y)- \Bigl(\frac{1}{s} - \tilde h_\Omega(x)\Bigr)\Phi_{s,y}(x) \right| = \left| \left[\frac{\mathbb{F}_s-id}{s} \Phi_{s, y}\right](x) + [L_{\Delta}\Phi_{s, y}](x)- [L_{\Delta}^\Omega\Phi_{s, y}](x) \right| \nonumber\\
&\leq \left|\left[\frac{\mathbb{F}_s-id}{s} \Phi_{s, y}\right](x)+[L_{\Delta}\Phi_{s, y}](x)\right|+\left|[L_{\Delta}^\Omega \Phi_{s, y}](x)\right|\quad \text{for $x \in \Omega$, $y \in \RR^N \setminus \Omega$.} \label{Eq1.10}
\end{align}
The first term on the RHS of the above inequality can be estimated using \eqref{eq:rn-case-second-claim} of Lemma \ref{Lem: lemma 4.1} with $\alpha = \frac{1}{2}$. Indeed, we have
\begin{align}\label{eq1.11}
\left|\left[\frac{\mathbb{F}_s-id}{s} \Phi_{s, y}\right](x)+[L_{\Delta}\Phi_{s, y}](x)\right| \leq \frac{sC}{\varepsilon \delta_{\Omega}(x)^{\varepsilon}}\|\Phi_{s, y}\|_{C^{\frac{1}{2}}(\Omega)}    
\end{align}
for $x\in \Omega$, $\varepsilon \in (0,1)$, $s \in (0, 1/2]$. Moreover, since
\begin{align}
  \Bigl||x_1-y|^{-N-2s}-|x_2-y|^{-N-2s}\Bigr| &\le (N+2s) |x_1-x_2|\Bigl(|x_1-y|^{-N-2s-1}+|x_1-y|^{-N-2s-1}\Bigr)\notag\\
   &\le 2(N+2s) |x_1-x_2|\delta_\Omega^{-N-2s-1}(y)\label{reuse}\\
                                              &\le 2(N+2s)r_{\Omega}\delta_\Omega^{-N-2s-1}(y) \qquad \text{for $x_1,x_2 \in \Omega$}\notag
\end{align}
by (\ref{alpha-est}) and 
\begin{equation}
  \label{sup-norm-est}
  \|\Phi_{s, y}\|_{L^\infty(\Omega)} \le c_{N,s}\delta_\Omega^{-N-2s}(y),    
\end{equation}
we have   
\begin{equation}
  \label{hölder-norm-est}
  \|\Phi_{s, y}\|_{C^{\frac{1}{2}}(\Omega)} \le  c_{N,s}\delta_\Omega^{-N-2s}(y)+  2(N+2s)r_{\Omega}\delta_\Omega^{-N-2s-1}(y).    
\end{equation}
The second term in \eqref{Eq1.10} can be estimated similarly with \eqref{reuse}. 
\begin{align*}
  |[L_{\Delta}^\Omega\Phi_{s, y}](x)|&\le c_{N,s}\int_{\Omega}\frac{\bigl| |x-y|^{-N-2s}-|z-y|^{-N-2s}\bigr|}{|x-z|^N}\,dz\\
  &\le  2(N+2s)  c_{N,s}\delta_\Omega^{-N-2s-1}(y) \int_{\Omega} |x-z|^{1-N}\,dz \\
  &\le  2(N+2s)  c_{N,s} \delta_\Omega(y)^{-N-2s-1} \int_{B_{r_{\Omega}}(x)} |x-z|^{1-N}\,dz\\
&=2 (N+2s)c_{N,s} \omega_{N-1} r_\Omega \delta_\Omega^{-N-2s-1}(y).
  \end{align*}
  Consequently, using \eqref{hölder-norm-est} in \eqref{eq1.11} and from the above estimate, we have
  \begin{align}
    \label{eq:I-final-est}
    &\left|\mathrm{I}(x,y)- \Bigl(\frac{1}{s} - \tilde h_\Omega(x)\Bigr)\Phi_{s,y}(x) \right|\\
    &\leq \frac{sC}{\varepsilon \delta_{\Omega}(x)^{\varepsilon}}
\Bigl(c_{N,s}\delta_\Omega^{-N-2s}(y)+  2(N+2s)r_{\Omega}\delta_\Omega^{-N-2s-1}(y)\Bigr)+ 2 (N+2s)c_{N,s} \omega_{N-1} r_\Omega \delta_\Omega^{-N-2s-1}(y). \nonumber
  \end{align}
\vspace{.5cm}
\noindent{\bf Estimate of $\mathrm{II}(x,y)$:} To estimate $\mathrm{II}(x,y)$, we use Lemma \ref{Lem: lemma 4.2} and (\ref{sup-norm-est}). Indeed, we have
\begin{align}
  |\mathrm{II}(x,y)| =  \frac{c_{N, s}}{s}\int_{\Omega} \frac{H_s(x, z)}{|y-z|^{N+2s}}\,dz&= \frac{1}{s}\left[\mathbb{H}_s \Phi_{s, y}\right](x) \nonumber\\
  &\leq sC\delta^{-\varepsilon}_{\Omega}(x)\|\Phi_{s, y}\|_{L^{\infty}(\Omega)} \le sc_{N,s} C\delta^{-\varepsilon}_{\Omega}(x) \delta^{-N-2s}_\Omega(y).  \label{II-estimate}
\end{align}
Combining (\ref{Eq1.10}), (\ref{eq:I-final-est}) and (\ref{II-estimate}), we therefore conclude that 
\begin{align}
  \label{difference-estimate}
&\left|\frac{P_s(x, y)}{s} - \Bigl(\frac{1}{s} - \tilde h_\Omega(x)\Bigr)\Phi_{s,y}(x) \right|\nonumber\\
    &\leq \frac{sC}{\varepsilon \delta_{\Omega}(x)^{\varepsilon}}
      \Bigl(c_{N,s}\delta_\Omega^{-N-2s}(y)+  2(N+2s)r_{\Omega}\delta_\Omega^{-N-2s-1}(y)\Bigr)\nonumber\\
  &\qquad + 2 (N+2s)c_{N,s}\omega_{N-1} r_\Omega \delta_\Omega^{-N-2s-1}(y) + sc_{N,s} C\delta^{-\varepsilon}_{\Omega}(x) \delta^{-N-2s}_\Omega(y)\nonumber\\
    &\leq  \frac{s c_{N,s}C}{\varepsilon \delta_{\Omega}(x)^{\varepsilon}} \delta^{-N-2s}_\Omega(y) + 
r_{\Omega} \Bigl(\frac{2 sC (N+2s)}{\varepsilon \delta_{\Omega}(x)^{\varepsilon}}+ 2(N+2s) \omega_{N-1} c_{N,s} \Bigr)\delta_\Omega^{-N-2s-1}(y),
\end{align}
where,  since $c_{N,s}=O(s)$, the right hand side tends to zero as $s \to 0^+$ uniformly in $x \in \Omega \setminus N_\beta$, $y \in B_r \setminus \Omega_{\beta}$.  This completes the proof of the claim \eqref{first claim}, since
$$
\Bigl(\frac{1}{s} - \tilde h_\Omega(x)\Bigr)\Phi_{s,y}(x)= \Bigl(\frac{c_{N,s}}{s} - c_{N,s}\tilde h_\Omega(x)\Bigr)\frac{1}{|x-y|^{N+2s}} \to \frac{c_N}{|x-y|^{N}} \qquad \text{as $s \to 0^+$}
$$
uniformly in $x \in \Omega \setminus N_\beta$, $y \in B_r \setminus \Omega_{\beta}$.
\end{proof}
\begin{corollary}
\label{beta-1-beta-2-r-cor}
If $g \in L^\infty(\R^N \setminus \Omega)$, then for $\beta_0,\beta_1>0$ and $r > \tilde r_\Omega$ we have 
\begin{equation*}
\int_{B_r \setminus \Omega_{\beta_1}}\frac{P_s(\cdot, y)}{s}g(y)\, dy \to
     \int_{B_r \setminus \Omega_{\beta_1}} \frac{c_Ng(y)}{|\cdot-y|^N}\,dy \qquad \text{as $s \to 0^+$ uniformly in $\Omega \setminus N_{\beta_0}$.}
\end{equation*}
\end{corollary}

\begin{proof}
  Lemma~\ref{local-Poisson-kernel-convergence} with $\beta=\min\{\beta_0,\beta_1\}$ yields that
  $$
\lim_{s \to 0^+} \frac{P_s(x, y)}{s}g(y) = \frac{c_Ng(y)}{|x-y|^N} \qquad \text{uniformly in $x \in \Omega \setminus N_\beta$, $y \in B_r \setminus \Omega_{\beta}$.}
$$
whereas $\Omega \setminus N_{\beta_0} \subset \Omega \setminus N_\beta$ and $B_r \setminus \Omega_{\beta_1} \subset B_r \setminus \Omega_{\beta}$. Hence the claim follows.
\end{proof}

Moreover, we have the following.
\begin{lemma}
\label{remainder-integral-uniform-est}  
If $g \in L^\infty(\R^N \setminus \Omega)$ satisfies (\ref{eq:first-order-exp-assumption}) with $c_g = 0$, then for every $\eps>0$ we have
\begin{equation}
  \label{ps-lim-r-infty}
\delta_\Omega^\eps(x) \int_{\RR^N \setminus B_{r}} \frac{P_s(x,y)g(y)}{s}dy \to 0 \qquad \text{as $r \to  \infty$ uniformly in $s \in (0,\frac{1}{2})$, $x \in \Omega$.}  
\end{equation}
\end{lemma}

\begin{proof}
Let $r> \tilde{r}_\Omega.$ Then we have
\begin{align*}
\left|\int_{\RR^N\setminus B_r}\frac{P_s(x, y)g(y)}{s}\, dy\right|
&\leq \left|\int_{\RR^N\setminus B_r}\left(\frac{P_s(x, y)}{s}-c_{N, s}\left(\frac{1}{s}-\tilde h_{\Omega}(x)\right)|y|^{-N-2s}\right)g(y)\, dy\right|\\
&+ \left|\int_{\RR^N\setminus B_r}c_{N, s}\left(\frac{1}{s}-\tilde h_{\Omega}(x)\right)|y|^{-N-2s}g(y)\,dy\right|\\
&\leq ||g||_{L^{\infty}(\RR^N\setminus B_r)}\int_{\RR^N\setminus B_r}\left|\frac{P_s(x, y)}{s}-c_{N, s}\left(\frac{1}{s}-\tilde h_{\Omega}(x)\right)|y|^{-N-2s}\right|\, dy\\
&+\omega_{N-1}\left|\frac{c_{N, s}}{s}-c_{N, s}\tilde h_{\Omega}(x)\right|\int_{r}^{\infty}\frac{|\tilde g(\tau)|}{\tau^{1+2s}}\, d\tau.
\end{align*}
To estimate the second integral above, we note that, since $\frac{c_{N, s}}{s}\to c_{N}$ as $s\to 0^+,$ using the assumption \eqref{eq:first-order-exp-assumption} with $c_g=0,$ we get
\begin{align*}
    \delta^{\eps}_{\Omega}(x)\left|\frac{c_{N, s}}{s}-c_{N, s}\tilde h_{\Omega}(x)\right|\int_{r}^{\infty}\frac{|\tilde g(\tau)|}{\tau}\, d\tau \to 0\quad \text{as}\quad r \to \infty
\end{align*}
uniformly in $s\in (0, \frac{1}{2})$, $x\in \Omega$, and for fixed every $\eps>0$, since $|\tilde h_{\Omega}(x)|=|h_{\Omega}+\rho_N|\leq C(1+|\ln(\delta_{\Omega}(x)|)$ by the explicit representation of $h_{\Omega}$ in \eqref{def-h-Omega}.

In the remainder of the proof we discuss the convergence of the first integral. For this, we use the following two inequalities. The first one is, 
\begin{align*}
\Bigl| |y|^{-N-2s}- |x-y|^{-N-2s}\Bigr| &\le (N+2s) |x| \Bigl(|y|^{-N-2s-1}+|x-y|^{-N-2s-1}\Bigr)\\
&\le 2 (N+2s) \tilde{r}_\Omega \delta_\Omega^{-N-2s-1}(y) \qquad \text{for $x \in \Omega$, $y \in \RR^N \setminus \Omega$,}  
\end{align*}
and the second one is, 
\begin{align}\label{tau-int-est}
    \int_{\RR^N\setminus B_{r}}\delta^{-N-\tau}_{\Omega}(y)\, dy\leq \frac{2^{N+\tau}}{\tau}\omega_{N-1}\frac{1}{r^{\tau}} \quad\text{ for $r>2\tilde r_{\Omega}$ and $\tau>0$.}
\end{align}
To see the latter inequality, note that for $|y|\geq r>2\tilde r_{\Omega}>r_{\Omega}$, we have $\delta_{\Omega}(y)\geq |y|-r_{\Omega}\geq \frac{|y|}{2}.$ Therefore, we conclude
\begin{equation*}
  \int_{\RR^N \setminus B_{r}} \delta_\Omega^{-N-\tau}(y)dy \le 2^{N+\tau} \int_{\RR^N \setminus B_{r}} |y|^{-N-\tau}dy= 2^{N+\tau} \omega_{N-1} \frac{r^{-\tau}}{\tau},
\end{equation*}
as stated in \eqref{tau-int-est}. With this we estimate, for $r>2\tilde{r}_{\Omega}$,
\begin{align}
&\int_{\RR^N \setminus B_{r}} \Bigl|\frac{P_s(x,y)}{s}-c_{N,s}\Bigl(\frac{1}{s} - \tilde h_\Omega(x)\Bigr) |y|^{-N-2s}\Bigr| |g(y)|dy\notag\\
  &\le \int_{\RR^N  \setminus B_{r}} \Bigl|\frac{P_s(x,y)}{s}-c_{N,s}\Bigl(\frac{1}{s} - \tilde h_\Omega(x)\Bigr) |x-y|^{-N-2s}\Bigr|dy\label{estimate remainder}\\
  &+ 2(N+2s) c_{N,s} \tilde{r}_\Omega \Bigl(\frac{1}{s} - \tilde h_\Omega(x)\Bigr)  \|g\|_{L^\infty(\RR^N \setminus \Omega)} \int_{\RR^N  \setminus B_{r}}  \delta^{-N-2s-1}_\Omega(y)dy,\notag          
\end{align}
where
$$
\int_{\RR^N \setminus B_{r}} \delta_\Omega^{-N-2s-1}(y) dy \le 
2^{N+2s + 1} \omega_{N-1} \frac{r^{-2s-1}}{2s+1}\le 
\frac{2^{N+4} \omega_{N-1}}{r} \to 0 \qquad \text{for $r \to \infty$}     
$$
uniformly in $s \in (0,\frac{1}{2})$. Since $c_{N,s}=O(s)$ and $\delta_{\Omega}^{\eps}(x)\tilde{h}_{\Omega}(x)$ is bounded in $\Omega$, we have
\begin{align*}
    2(N+2s) c_{N,s} \tilde{r}_\Omega \Bigl(\frac{1}{s} - \tilde h_\Omega(x)\Bigr)  \|g\|_{L^\infty(\RR^N \setminus \Omega)} \int_{\RR^N  \setminus B_{r}}  \delta^{-N-2s-1}_\Omega(y)dy\to 0 \quad\text{as $r\to \infty$}
\end{align*}
uniformly for $s\in(0,\frac12)$ and almost uniformly in $x\in \Omega$. Thus, it remains to discuss the convergence of \eqref{estimate remainder}. For this, note that, for $r > 2 \tilde r_\Omega $, using (\ref{difference-estimate}), (\ref{tau-int-est}), and \[c_{N, s}\leq \frac{\Gamma(\frac{N}{2}+1)}{\pi^{\frac{N}{2}}} \quad \text{for $s \in (0, \frac{1}{2})$},\] we find that 
\begin{align*}
  &\int_{\RR^N  \setminus B_{r}} \Bigl|\frac{P_s(x,y)}{s}-c_{N,s}\Bigl(\frac{1}{s} - \tilde h_\Omega(x)\Bigr) |x-y|^{-N-2s}\Bigr|dy\\
  &\le \frac{s c_{N,s}C}{\varepsilon \delta_{\Omega}(x)^{\varepsilon}} \int_{\RR^N  \setminus B_{r}} \delta^{-N-2s}_\Omega(y)dy
    + r_{\Omega} \Bigl(\frac{2 sC (N+2s)}{\varepsilon \delta_{\Omega}(x)^{\varepsilon}}+ 4 \omega_{N-1} c_{N,s} \Bigr)\int_{\RR^N  \setminus B_{r}} \delta_\Omega^{-N-2s-1}(y)dy\\
  &\le \frac{s c_{N,s}C}{\varepsilon \delta_{\Omega}(x)^{\varepsilon}}
 \frac{2^{N+2s}}{2s}\omega_{N-1}r^{-2s} 
    + r_{\Omega} \Bigl(\frac{2 s C (N+2s)}{\varepsilon \delta_{\Omega}(x)^{\varepsilon}}+ 4 \omega_{N-1} c_{N,s} \Bigr)2^{N+2s+1} \frac{r^{-2s-1}}{2s+\frac{1}{2}}\\
  &\le   \frac{2^{N}C}{\varepsilon \delta_{\Omega}(x)^{\varepsilon}}\omega_{N-1} \bigl(c_{N,s}r^{-2s}\bigr)
    + r_{\Omega} \Bigl(\frac{ C (N+1)}{\varepsilon \delta_{\Omega}(x)^{\varepsilon}}+ 4\omega_{N-1} \frac{\Gamma(\frac{N}{2}+1)}{\pi^{\frac{N}{2}}} \Bigr)2^{N+3} \frac{1}{r}.
\end{align*}
Since
$$
c_{N,s}r^{-2s}  \to 0 \qquad \text{and}\qquad \frac{1}{r} \to 0 \qquad \text{as $r \to \infty$ uniformly in $s \in (0,\frac{1}{2})$,}
$$
we conclude the proof of (\ref{ps-lim-r-infty}).  
\end{proof}

We may now complete the
\begin{proof}[Proof of Theorem~\ref{thm1.1}(ii)]
  Let $g \in L^\infty(\R^N \setminus \Omega)$ and $\eps>0$. Without loss of generality, by replacing $g$ with $g-c_g$, we may assume that $c_g=0$. We then need to show that
  \begin{equation*}
    \label{eq:sufficient-theorem1-1-ii-0}
\delta^{\varepsilon}_{\Omega}(\cdot)\frac{u_s}{s} = \delta^{\varepsilon}_{\Omega}(\cdot) \int_{\R^N \setminus \Omega}\frac{P_s(\cdot, y)}{s}g(y)\, dy \to \delta^{\varepsilon}_{\Omega}(\cdot) \lim_{r \to \infty}\int_{B_r \setminus \Omega} \frac{c_Ng(y)}{|\cdot-y|^N}\,dy \quad
  \text{as $s \to 0^+$ uniformly on $\Omega$.}
\end{equation*}
By Lemmas~\ref{regarding-def-L-g-pv-sense} and~\ref{remainder-integral-uniform-est}, it suffices to show that for every fixed $r>\tilde r_\Omega$ we have 
  \begin{equation}
    \label{eq:sufficient-theorem1-1-ii}
    \delta^{\varepsilon}_{\Omega}(\cdot) \int_{B_r \setminus \Omega}\frac{P_s(\cdot, y)}{s}g(y)\, dy \to \delta^{\varepsilon}_{\Omega}(\cdot)
    \int_{B_r \setminus \Omega} \frac{c_Ng(y)}{|\cdot-y|^N}\,dy \quad \text{as $s \to 0^+$ uniformly on $\Omega$.}
\end{equation}
For this we let $s \in (0,\frac{1}{2})$ and first recall that, see \cite[Remark 3.8]{Chen99},
\begin{align}\label{EQ4.10}
&\left|\frac{P_s(x,y)g(y)}{s}\right| \leq C(1-s)\frac{|g(y)|}{\delta^s_{\Omega}(y)|x-y|^N}
\end{align}
for $x \in \Omega$, $y \in \R^N \setminus \Omega$, and a constant $C$ independent of $x,y,s$. Moreover, we let $\beta_0= \beta_0(\eps)$ be chosen such that $t \mapsto t^\eps |\log t|$ is increasing on $(0,\beta_0)$. Moreover, we let $x \in N_{\beta_0} \cap \Omega$ and estimate, using Lemma \ref{int:la},
\begin{align*}
  &\delta^{\varepsilon}_{\Omega}(x) \int_{B_r \setminus \Omega}\frac{P_s(x, y)}{s}g(y)\, dy\\
  &\leq C(1-s)\delta^{s+\varepsilon}_{\Omega}(x) \|g\|_{L^{\infty}(\R^N \setminus \Omega)}\Bigl( \int_{B_{r_{\Omega}(x)}\setminus \Omega}\frac{dy}{\delta^s_{\Omega}(y)|x-y|^N}+\int_{B_r \setminus B_{r_{\Omega}(x)}}\frac{1}{|x-y|^{N}}\,dy\Bigr)\\
  &\leq C(1-s)\delta^{s+\varepsilon}_{\Omega}(x)\Bigl(\frac{1}{1-s}\left[1+\frac{\min\left\{\frac{1}{s}, 1+|\log\delta_{\Omega}(x)|\right\}}{\delta^s_{\Omega}(x)}\right]+ \int_{B_{2r} \setminus B_{r_{\Omega}}}|y|^{-N}\,dy\Bigr)\\
&\leq C\delta^{\varepsilon}_{\Omega}(x)\Bigl(2+|\log \delta_{\Omega}(x)|+\omega_{N-1}\log\frac{2r}{r_\Omega}\Bigr)\leq C\beta_0^{\varepsilon}\Bigl(2+|\log \beta_0|+\omega_{N-1}\log\frac{2r}{r_\Omega}\Bigr).
\end{align*}
with $C>0$ independent of $x,s$, where the RHS tends to zero as $\beta_0 \to 0$. Hence, if $\delta>0$ is given, we may fix $\beta_0>0$ small enough to guarantee
\begin{equation}\label{beta0 a}
\delta^{\varepsilon}_{\Omega}(x) \int_{B_r \setminus \Omega}\frac{P_s(x, y)}{s}g(y)\, dy < \delta \qquad \text{for $x \in N_{\beta_0} \cap \Omega$.}
\end{equation}
By a similar but easier argument we may shrink $\beta_0$, if necessary, such that we also have
\begin{equation}\label{beta0 b}
\delta^{\varepsilon}_{\Omega}(x) \int_{B_r \setminus \Omega}\frac{c_N}{|x-y|^N}|g(y)|\, dy \le \delta \quad \text{for $x \in N_{\beta_0} \cap \Omega$.}
\end{equation}

Now we fix $\beta_1 \in (0,\beta_0)$. For $x \in \Omega \setminus N_{\beta_0}$, we note that \eqref{EQ4.10} yields
\begin{align*}
 \delta^{\varepsilon}_{\Omega}(x) \int_{N_{\beta_1} \setminus \Omega}\frac{P_s(x, y)}{s}g(y)\, dy
  &\leq C(1-s)\delta^{s+\varepsilon}_{\Omega}(x) \|g\|_{L^{\infty}(\R^N \setminus \Omega)}\int_{N_{\beta_1}\setminus \Omega}\frac{dy}{\delta^s_{\Omega}(y)|x-y|^N}dy\\
  &\leq \frac{C(1-s)}{\beta_0^N} \delta^{s+\varepsilon}_{\Omega}(x) \|g\|_{L^{\infty}(\R^N \setminus \Omega)} \int_{N_{\beta_1}\setminus \Omega}\frac{dy}{\delta^s_{\Omega}(y)}dy\\
    &\leq \frac{C (\text{diam}\Omega)^{s+\eps}}{\beta_0^N} \|g\|_{L^{\infty}(\R^N \setminus \Omega)} \int_{N_{\beta_1}\setminus \Omega}\frac{dy}{\delta^{1/2}_{\Omega}(y)}dy 
\end{align*}
for $s \in (0,\frac{1}{2})$, where, by monotone convergence and the regularity of $\partial \Omega$,  
$$
\int_{N_{\beta_1}\setminus \Omega}\frac{dy}{\delta^{1/2}_{\Omega}(y)}dy \to 0 \qquad \text{as $\beta_1 \to 0$.}
$$
Hence, we may fix $\beta_1 \in (0,\beta_0)$ depending on $\beta_0$ and $g$ with  
\begin{equation}\label{beta1 a}
\delta^{\varepsilon}_{\Omega}(x) \int_{N_{\beta_1} \setminus \Omega}\frac{P_s(x, y)}{s}g(y)\, dy\le  \frac{C}{\beta_0^N} \|g\|_{L^{\infty}(\R^N \setminus \Omega)} \int_{N_{\beta_1}\setminus \Omega}\frac{dy}{\delta^{1/2}_{\Omega}(y)}dy \le \delta \qquad \text{for $x \in \Omega \setminus N_{\beta_0}$.}
\end{equation}
A similar argument allows us to shrink $\beta_1$ further, if necessary, such that
\begin{equation}\label{beta1 b}
\delta^{\varepsilon}_{\Omega}(x) \int_{N_{\beta_1} \setminus \Omega}\frac{c_N}{|x-y|^N}|g(y)|\, dy \le \delta \quad \text{for $x \in \Omega \setminus N_{\beta_0}.$} 
\end{equation}
Consequently, \eqref{beta0 a}, \eqref{beta0 b} yield
\begin{equation}
  \label{eq:delta-diff-est-1}
\sup_{x \in N_{\beta_0} \cap \Omega} \delta^\eps_\Omega(x) \Bigl|\int_{B_r \setminus \Omega}\frac{P_s(x, y)}{s}g(y)\, dy- \int_{B_r \setminus \Omega}\frac{c_N}{|x-y|^N}|g(y)|\, dy\Bigr| \le 2 \delta
\end{equation}
and \eqref{beta1 a}, \eqref{beta1 b} yield
\begin{equation}
  \label{eq:delta-diff-est-2}
\sup_{x \in \Omega \setminus N_{\beta_0}} \delta^\eps_\Omega(x) \Bigl|\int_{N_{\beta_1} \setminus \Omega}\frac{P_s(x, y)}{s}g(y)\, dy- \int_{N_{\beta_1} \setminus \Omega}\frac{c_N}{|x-y|^N}|g(y)|\, dy\Bigr| \le 2 \delta
\end{equation}
for every $s \in (0,\frac{1}{2})$. Combining (\ref{eq:delta-diff-est-2}) with Corollary~\ref{beta-1-beta-2-r-cor} gives that
\begin{equation*}
  \label{eq:delta-diff-est-3}
\limsup_{s \to 0^+} \sup_{x \in \Omega \setminus N_{\beta_0}} \delta^\eps_\Omega(x) \Bigl|\int_{B_r \setminus \Omega}\frac{P_s(x, y)}{s}g(y)\, dy- \int_{N_{\beta_1} \setminus \Omega}\frac{c_N}{|x-y|^N}|g(y)|\, dy\Bigr| \le 2 \delta,
\end{equation*}
and together with (\ref{eq:delta-diff-est-1}) this implies that
$$
\limsup_{s \to 0^+} \sup_{x \in \Omega} \delta^\eps_\Omega(x) \Bigl|\int_{B_r \setminus \Omega}\frac{P_s(x, y)}{s}g(y)\, dy- \int_{N_{\beta_1} \setminus \Omega}\frac{c_N}{|x-y|^N}|g(y)|\, dy\Bigr| \le 2 \delta.
$$
Since $\delta>0$ was chosen arbitrarily, (\ref{eq:sufficient-theorem1-1-ii}) follows, and the proof is finished.
\end{proof}

We may now also complete the

\begin{proof}[Proof of Corollary~\ref{cor-harnack}]
We note that by Theorem~\ref{thm1.1} we have 
\begin{align*}
    \lim_{s\to 0}\frac{u_s(x)}{s}&=c_N\int_{\RR^N\setminus \Omega}\frac{g(z)}{|x-z|^N}\,dz \leq c_N\int_{\RR^N\setminus \Omega}\frac{g^+(z)}{|x-z|^N}\,dz\\
                                 &\leq c(x,y) c_N\int_{\RR^N\setminus \Omega}\frac{g^+(z)}{|y-z|^N}\,dz = c(x,y) \Bigl(\lim_{s\to 0}\frac{u_s(y)}{s} + c_N\int_{\RR^N\setminus \Omega}\frac{g^-(z)}{|y-z|^N}\,dz\Bigr)
    \end{align*}
    with
    $$
    c(x,y) = \Bigl(\sup_{z \in \RR^N\setminus \Omega} \frac{|y-z|}{|x-z|}\Bigr)^N \le \diam \Bigl(\frac{\diam(\Omega)}{\delta_\Omega(x)}\Bigr)^N \quad \text{for $x,y \in \Omega$,}
    $$
and thus (\ref{harnack}) follows. 
\end{proof}

\section{A counterexample}
\label{sec:counterexample}

In this section, we construct the counterexample that yields the proof of Theorem~\ref{thm-limsup-liminf}. 
For this, let $\Omega = B_1(0) \subset \R^N$, $N \ge 2$, and we construct a boundary data $g \in L^\infty(\R^N \setminus \Omega)$ (independent of $s$) with the property that the unique solutions $u_s$ of (\ref{main eq}) satisfy $0 \le u_s \le 1$ in $\Omega$ for $s \in (0,1)$ and  (\ref{eq:limsup-liminf}). We construct $g$ in the form 
$$
g(x)=\frac{|x|^2f(|x|^2-1)}{|x|^2-1}  \qquad \text{for $|x|>1$},
$$
where $f\in L^{\infty}([0,\infty))$ is a function with $f \equiv 0$ in a right neighborhood of zero. Then we obviously have $g \in L^\infty(\R^N \setminus \Omega)$, and the unique solutions $u_s$ of (\ref{main eq}) satisfy assumption~(\ref{eq:them1.0-assumption}). Therefore, by Theorem~\ref{thm1.0}, we only need to find $f$ with the property that
  \begin{equation*}
    \label{eq:limsup-liminf-0}
  \limsup_{s \to 0^+}u_s(0)=1, \quad \liminf_{s \to 0^+}u_s(0)=0. 
  \end{equation*}
By (\ref{eq:poisson-rep-balls}), we have, by a change of variable,
\begin{align*}
  u_s(0)&= \gamma_{N,s}\int_{\RR^N\setminus B_1(0)}\frac{(1-|x|^2)^s}{(|y|^2-1)^s|y|^N}g(y)\, dy\\
&\omega_{N-1}\gamma_{N,s} \int_1^{\infty}\frac{r f(r^2-1)}{(r^2-1)^{s+1}}\, dr= \frac{\omega_{N-1}\gamma_{N,s}}{2s} s \int_0^{\infty}\frac{f(t)}{t^{1+s}}\, dt,
\end{align*}
where $\gamma_{N,s}=\frac{\Gamma(\frac{N}{2})}{\pi^{\frac{N}{2}}\Gamma(s)\Gamma(1-s)} = \frac{2s}{\omega_{N-1}\Gamma(1+s)\Gamma(1-s)}$  and therefore
$$
\frac{\omega_{N-1}\gamma_{N,s}}{2s} = \frac{1}{\Gamma(1+s)\Gamma(1-s)} \to 1 \qquad \text{as $s \to 0^+$.}
$$
Hence, we need to find $f$ with
  \begin{equation}
    \label{eq:limsup-liminf-f}
  \limsup_{s \to 0^+}s \int_0^{\infty}\frac{f(t)}{t^{1+s}}\, dt =1 \qquad \text{and}\qquad  \quad \liminf_{s \to 0^+} s \int_0^{\infty}\frac{f(t)}{t^{1+s}}\, dt =0. 
  \end{equation}
We choose $f$ of the form 
\begin{equation}
  \label{eq:f-special-choice}
f(t)= \sum_{k=1}^{\infty} 1_{[r_{2k-1},r_{2k}]}(t),
\end{equation}
where $(r_k)_k$ is a strictly increasing sequence with $r_1=1$. Since $0 \le f \le 1_{[1,\infty)}$, we have
\begin{equation}
  \label{eq:bounds-f-integral}
0 \le s \int_0^{\infty}\frac{f(t)}{t^{1+s}}\, dt \le s \int_1^\infty t^{-1-s}dt = 1 \qquad \text{for every $s \in (0,1)$.}
\end{equation}
Moreover, (\ref{eq:f-special-choice}) yields that 
$$
s \int_0^{\infty}\frac{f(t)}{t^{1+s}}\, dt = \sum_{k=1}^\infty\Bigl(\frac{1}{r_{2k-1}^s}- \frac{1}{r_{2k}^s}\Bigr)= \sum_{j=1}^\infty (-1)^{j+1} r_j^{-s}.
$$
We now choose $r_k$ and $s_k$, $k \in \N$ inductively to guarantee that  
\begin{equation}
  \label{eq:sequence-est}
\lim_{k \to \infty} \sum_{j=1}^\infty (-1)^{j+1} r_j^{-s_{2k}}= 0 \qquad \text{and}\qquad \lim_{k \to \infty} \sum_{j=1}^\infty (-1)^{j+1} r_j^{-s_{2k-1}}= 1
\end{equation}
We recall that $r_1 = 1$ has already been chosen and we let $s_1=1$. Next, we let $(\eps_k)_k$ be a sequence of positive numbers with $\eps_k \to 0$ as $k \to \infty$. Assuming that $r_1,\dots,r_{k-1}$ and $s_1,\dots,s_{k-1}$ have already been chosen for some $k \ge 2$, we choose $r_k$ and $s_k$ by a case distinction.

If $k$ is odd, we choose $r_k> r_{k-1}$ sufficiently large to ensure that
$$
r_{k}^{-s_{k-1}}< \frac{\eps_{k-1}}{2}.
$$
Since $\lim \limits_{s \to 0^+}\sum\limits_{j=1}^{k}(-1)^{j+1} r_j^{-s} \to \sum\limits_{j=1}^{k}(-1)^{j+1} = 1$, we may choose $s_k \in (0,1)$ sufficiently small to guarantee that
$$
\sum_{j=1}^{k}(-1)^{j+1} r_j^{-s_{k}}> 1-\frac{\eps_k}{2}. 
$$
If $k$ is even, we choose $r_k > r_{k-1}$ sufficiently large to ensure that 
$$
r_{k}^{-s_{k-1}}< \frac{\eps_{k-1}}{2}.
$$
Similarly as in the first case, since $\lim \limits_{s \to 0^+}\sum\limits_{j=1}^{k}(-1)^{j+1} r_j^{-s} \to \sum\limits_{j=1}^{k}(-1)^{j+1} = 0$, we may choose $s_k \in (0,1)$ sufficiently small to guarantee that
$$
\sum_{j=1}^{k}(-1)^{j+1} r_j^{-s_{k}}< \frac{\eps_k}{2}.
$$
We may now estimate, for $k$ odd, that
$$
\sum_{j=1}^\infty (-1)^{j+1} r_j^{-s_{k}} > 1-\frac{\eps_k}{2} + \sum_{j=k+1}^{\infty}(-1)^{j+1} r_j^{-s_{k}}
\ge 1-\frac{\eps_k}{2} - r_{k+1}^{-s_k} \ge 1 - \eps_k, 
$$
whereas for $k$ even we have 
$$
\sum_{j=1}^\infty (-1)^{j+1} r_j^{-s_{k}} < \frac{\eps_k}{2} + \sum_{j=k+1}^{\infty}(-1)^{j+1} r_j^{-s_{k}}
\le 1-\frac{\eps_k}{2} + r_{k+1}^{-s_k} \le \eps_k. 
$$
Combining these estimates with (\ref{eq:bounds-f-integral}), we have proved~(\ref{eq:sequence-est}), and therefore (\ref{eq:limsup-liminf-f}) holds.

\section{Differentiability of the solution for fixed \texorpdfstring{$s>0$}{s>0}}\label{diff for any s}
In this section, we prove \cref{thm1.3} and Corollary~\ref{cor-thm-1-3}. For this we need to collect some preliminary tools.
The following convergence result is taken from \cite{JSW20}.
\begin{lemma}[Lemma 5.2, \cite{JSW20}]\label{lemma:continuity-smooth}
Let $s_0 \in (0,1]$, $\delta\in(0,\frac{s_0}{2})$, let $(s_n)_n\subset (s_0-\delta,\max\{1,s_0+\delta\}]$ be a sequence with $s_n \to s_0$, 
Moreover, let $(f_n)_n \subset L^\infty(\Omega)$ be a sequence with $f_n \to f_0 \in L^\infty(\Omega)$. Then we have
\[
\cGs_{s_n} f_n \to \cGs_{s_0} f_0 \quad\text{in $C^{s_0-2\delta}_0(\Omega)$ as $s\to s_0,$}
\]
where $\cGs_{(\cdot)}$ is the Green operator defined in \eqref{eq:G-op-definition}.
\end{lemma}

We also need the following extension of \cite[Theorem 1.1]{CW19}.  

\begin{lemma}
  \label{extension-CW-19}
Let $f \in L^1_0(\R^N)$ with $f \in C^\gamma(U)$ for some $\gamma>0$ and some open neighborhood $U \subset \R^N$ of $\overline{\Omega}$. Then we have 
  \begin{equation}
    \label{diff-f-L-1-0}
  \lim_{\sigma \to 0^+} \Bigl\|\frac{(-\Delta)^\sigma f-f}{\sigma}-\loglap f \Bigr\|_{L^\infty(\Omega)} = 0.
  \end{equation}
\end{lemma}

\begin{proof}
Let $\phi\in C^{\infty}_c(U) \subset C^\infty_c(\R^N)$ be a function with $0\leq \phi\leq 1$ and $\phi\equiv 1$ in an open neighborhood $\tilde U \subset U$ of $\overline{\Omega}$. Then, writing
$$
f=f \phi +f\left(1-\phi\right)
$$
we have  $f \phi \in C^{\gamma}_c(\RR^N)$. Hence, using \cite[Theorem 1.1]{CW19}, we obtain
\begin{align*}
 \left(\frac{(-\Delta)^{\sigma}-\id}{\sigma}\right)f\phi\to L_{\Delta}f\phi\,\,\, \text{as}\,\,\, \sigma \to 0\,\,\, \text{in}\,\,\, L^{\infty}(\RR^N).  
\end{align*}
Now, to show
\begin{equation}
\label{asdfg}
\left(\frac{(-\Delta)^{\sigma}-\id}{\sigma}\right)f(1-\phi) \to  L_{\Delta}f(1-\phi)\,\,\, \text{as}\,\,\, \sigma\to 0\,\,\, \text{in}\,\,\, L^{\infty}(\Omega),   
\end{equation}
note that for $x \in \Omega,$ we have 
\begin{equation*}\label{asdfg-2}
  \left[\left(\frac{(-\Delta)^{\sigma}-\id}{\sigma}\right)f(1-\phi)\right](x)=
  \left[\left(\frac{(-\Delta)^{\sigma}}{\sigma}\right)f(1-\phi)\right](x)= \frac{c_{N, \sigma}}{\sigma}\int_{\RR^N\setminus \tilde U}\frac{f(\phi-1)(y)}{|x-y|^{N+2\sigma}}\,dy 
\end{equation*}
and 
\begin{equation*}\label{asdfg-1}
  \left[L_{\Delta}f(1-\phi)\right](x)= c_N \int_{\RR^N\setminus \Omega_0}\frac{f(\phi-1)(y)}{|x-y|^{N}}\,dy. 
\end{equation*}
Hence,
\begin{align}
\left|  \Bigl[\Bigl(\frac{(-\Delta)^{\sigma}-\id}{\sigma}\Bigr)f(1-\phi)\Bigr](x)-  \left[L_{\Delta}f(1-\phi)\right](x) \right| \le \int_{\RR^N\setminus \tilde U}\Bigl|\frac{c_{N, \sigma}}{\sigma |x-y|^{N+2\sigma}}- \frac{c_N}{|x-y|^{N}}\Bigr||f(y)| \,dy.  \label{asdfg-3}
\end{align}
Next, we let $\eps>0$ and set $B_r:= B_r(0)$ for $r>0$. Since $f \in L^1_0(\R^N)$, we may choose $R>1$ sufficiently large with $\Omega \subset \tilde U \subset B_{R/2}$ and such that 
\begin{equation}
  \label{eq:remainder-f-s-L-1}
\int_{\R^N \setminus B_R} \frac{|f(y)|}{(1+|y|)^N} \,dy \le \eps.
\end{equation}
We note that,
$$
|x-y| \ge |y|-|x| \ge |y|- \frac{R}{2} \ge \frac{|y|}{2} \qquad \text{for $x \in \Omega$, $y \in \R^N \setminus B_R$},
$$
and therefore from \eqref{eq:remainder-f-s-L-1}, we have
\begin{align*}
  & \int_{\RR^N\setminus B_R}\Bigl|\frac{c_{N, \sigma}}{\sigma |x-y|^{N+2\sigma}}- \frac{c_N}{|x-y|^{N}}\Bigr||f(y)| \,dy\\
  &\le \int_{\R^N \setminus B_R} \frac{|f(y)|}{(1+|y|)^N} \,dy\: \Bigl\| \frac{c_N (1+|\cdot|)^{N}}{|x-(\cdot)|^{N}} \Bigl(\frac{c_{N, \sigma}}{\sigma c_N } |x-(\cdot)|^{-2\sigma}- 1\Bigr)\Bigr\|_{L^\infty(\R^N \setminus B_R)}\\
  &\le \eps \Bigl\| 2^N c_N \frac{(1+|\cdot|)^{N}}{|\cdot|^{N}} \Bigl(\frac{c_{N, \sigma}}{c_N \sigma} 2^{2\sigma}|\cdot |^{-2\sigma} + 1\Bigr)\Bigr\|_{L^\infty(\R^N \setminus B_R)} \le \eps  4^{N} c_N  \Bigl(\frac{c_{N, \sigma}}{\sigma c_N }4^{\sigma} + 1\Bigr) \qquad \text{for $x \in \Omega$,}
\end{align*}
which, since $\frac{4^{\sigma}c_{N, \sigma}}{\sigma c_N }\to 1$ as $\sigma \to 0^+$, implies that
\begin{equation}
  \label{eq:eps-R-est-1}
\limsup_{\sigma \to 0^+}\sup_{x \in \Omega} \int_{\RR^N\setminus B_R}\Bigl|\frac{c_{N, \sigma}}{\sigma |x-y|^{N+2\sigma}}- \frac{c_N}{|x-y|^{N}}\Bigr||f(y)| \,dy \le \eps  2^{2N+1} c_N.
\end{equation}
Moreover, since $\tau:=\text{dist} (\Omega, \R^N \setminus \tilde U)>0$, we have
$$
\frac{c_{N, \sigma}}{\sigma c_N } |x-y|^{-2\sigma} \to 1 \qquad \text{as $\sigma \to 0^+$ uniformly in $x \in \Omega$, $y \in B_R \setminus \tilde U,$}
$$
and also 
\begin{align*}
&\int_{B_R \setminus \tilde U}\Bigl|\frac{c_{N, \sigma}}{\sigma |x-y|^{N+2\sigma}}- \frac{c_N}{|x-y|^{N}}\Bigr||f(y)| \,dy\\  
  &\le \int_{\R^N} \frac{|f(y)|}{(1+|y|)^N} \,dy  \Bigl\| \frac{c_N (1+|\cdot|)^{N}}{|x-(\cdot)|^{N}} \Bigl(\frac{c_{N, \sigma}}{c_N \sigma} |x-(\cdot)|^{-2\sigma}- 1\Bigr)\Bigr\|_{L^\infty(B_R \setminus \tilde U)}\\   
  &\le \frac{c_N (R+1)^{N}}{\tau^{N}}  \int_{\R^N} \frac{|f(y)|}{(1+|y|)^N} \,dy
    \Bigl\|  \Bigl(\frac{c_{N, \sigma}}{\sigma c_N } |x-(\cdot)|^{-2\sigma}- 1\Bigr)\Bigr\|_{L^\infty(B_R \setminus \tilde U)},  
\end{align*}
which implies that 
\begin{equation}
  \label{eq:eps-R-est-2}
\lim_{\sigma \to 0^+}\sup_{x \in \Omega} \int_{B_R \setminus \tilde U}\Bigl|\frac{c_{N, \sigma}}{\sigma |x-y|^{N+2\sigma}}- \frac{c_N}{|x-y|^{N}}\Bigr||f(y)| \,dy =0. 
\end{equation}
Combining (\ref{eq:eps-R-est-1}) and (\ref{eq:eps-R-est-2}) gives 
$$
\limsup_{\sigma \to 0^+}\sup_{x \in \Omega} \int_{\RR^N\setminus \tilde U}\Bigl|\frac{c_{N, \sigma}}{\sigma |x-y|^{N+2\sigma}}- \frac{c_N}{|x-y|^{N}}\Bigr||f(y)| \,dy \le \eps  2^{2N+1} c_N
$$
for every $\eps>0$ and therefore 
\begin{equation*}
  \label{eq:eps-R-est-3}
\lim_{\sigma \to 0^+}\sup_{x \in \Omega} \int_{\RR^N\setminus \tilde U}\Bigl|\frac{c_{N, \sigma}}{\sigma |x-y|^{N+2\sigma}}- \frac{c_N}{|x-y|^{N}}\Bigr||f(y)| \,dy = 0.
\end{equation*}
Consequently, (\ref{asdfg}) follows by (\ref{asdfg-3}). This complete the proof of \eqref{diff-f-L-1-0}.
\end{proof}

Next, we need the following observation.

\begin{lemma}
  \label{observation-pointwise}
  Let $0<s<\frac{\alpha}{2}<1$ and $g \in L^1_s(\R^N) \cap C^\alpha_{\loc}(\Omega)$.
Then we have 
\begin{equation}
  \label{eq:consistency-definitions}
\int_{\Omega}[(-\Delta)^sg] \phi\,dx =  \int_{\R^N}g (-\Delta)^s \phi\,dx \qquad \text{for all $\phi \in C^\infty_c(\Omega)$.}  
\end{equation}
\end{lemma}
We note that this lemma guarantees that the pointwise principle value definition of $(-\Delta)^sg$ in $\Omega$ is consistent with the definition of $(-\Delta)^s g$ as a distribution on $\R^N$ via
\begin{equation*}
  \label{eq:distr-definition}
    \int_{\R^N}[(-\Delta)^s g]  \phi\,dx :=   \int_{\R^N}g (-\Delta)^s \phi\,dx \qquad \text{for all $\phi \in C^\infty_c(\R^N)$.}
\end{equation*}

\begin{proof}
For $\phi \in C^\infty_c(\Omega)$ we have, with $K:=\supp \phi \subset \subset \Omega$, by the Fubini theorem and a change of variable, 
\begin{align*}
&\int_{\R^N}g (-\Delta)^s \phi\,dx = \frac{c_{N,s}}{2}\int_{\R^N}g(x)\int_{\R^N}\frac{2\phi(x)-\phi(x+z)-\phi(x-z)}{|z|^{N+2s}}dz dx\\
                                    &\frac{c_{N,s}}{2}\int_{\R^N}g(x)\int_{\R^N}\frac{2\phi(x)-\phi(x+z)-\phi(x-z)}{|z|^{N+2s}}dx dz\\
                                    & = \frac{c_{N,s}}{2}\int_{\R^N}\frac{1}{|z|^{N+2s}}\Bigl( \int_{\R^N} 2\phi(x)g(x)dx - \int_{\R^N}g(x)\phi(x+z)\bigr)dx -\int_{\R^N}\bigl(g(x) \phi(x-z)\bigr)dx\Bigr)dz\\ 
& = \frac{c_{N,s}}{2}\int_{\R^N}\frac{1}{|z|^{N+2s}}\Bigl( \int_{\R^N} 2\phi(x)g(x)dx - \int_{\R^N}g(x-z)\phi(x)\bigr)dx -\int_{\R^N}\bigl(g(x+z) \phi(x)\bigr)dx\Bigr)dz\\ 
& = \frac{c_{N,s}}{2}\int_{\R^N}\Bigl( \int_{K} \phi(x)\frac{2 g(x)-g(x+z)-g(x-z)}{|z|^{N+2s}}dx \Bigr)dz\\ 
& = \frac{c_{N,s}}{2}\int_{K} \phi(x) \int_{\R^N}\frac{2 g(x)-g(x+z)-g(x-z)}{|z|^{N+2s}}dz=\int_{\Omega} \phi (-\Delta)^sg\,dx.
\end{align*}
\end{proof}

With the help of Lemma~\ref{observation-pointwise}, we may now establish the following local version of the semigroup property for fractional Laplacians.
\begin{proposition}
\label{semigroup-general}
Let $0<t\le \frac{s}{2}<s<\frac{\alpha}{2}<1$, and let $g \in L^1_{t}(\R^N) \cap C^\alpha_{\loc}(\Omega)$ satisfy $(-\Delta)^{s-t} g \in L^1_{t}(\R^N)$. Then we have 
\begin{equation}
\label{semigroup-simple}
  (-\Delta)^s g = (-\Delta)^{t}(-\Delta)^{s-t} g \qquad \text{in $\Omega$.}
\end{equation}
\end{proposition}

\begin{proof}
We start by noting that,  by the assumption $(-\Delta)^{s-t} g \in L^1_{t}(\R^N)$ and since $(-\Delta)^{s-t} g \in C^{\alpha-2(s-t)}_{\loc}(\Omega)$, both sides of (\ref{semigroup-simple})
are well-defined in pointwise sense in $\Omega$. It then suffices to show that
  \begin{equation}
    \label{eq:distr-sufficient}
 \int_{\Omega}[(-\Delta)^s g]  \phi\,dx =  \int_{\Omega} [(-\Delta)^{s-t}(-\Delta)^t g] \phi\,dx \qquad \text{for all $\phi \in C^\infty_c(\Omega)$.}  
  \end{equation}
  By (\ref{eq:consistency-definitions}) and since $L^1_{t}(\R^N) \subset L^1_{s}(\R^N)$, we have
  \begin{equation}
    \label{first-eq-chain-semigroup}
   \int_{\Omega}[(-\Delta)^s g]  \phi\,dx = \int_{\R^N} g (-\Delta)^s \phi\,dx = \int_{\R^N}g (-\Delta)^{s-t} \psi \,dx
  \end{equation}
  with $\psi: = (-\Delta)^{t}\phi  \in C^\infty(\R^N) \cap L^\infty(\R^N,(1+|x|)^{N+2t}dx).$
  Moreover, we note that for any multiindex $\alpha$ we have $\partial^\alpha \phi \in C^\infty_c(\Omega)$, and
  $$
  \partial^{\alpha} \psi = (-\Delta)^s[\partial^{\alpha}\phi] \in L^\infty(\R^N,(1+|x|)^{N+2t}dx).
  $$
  In particular, $\psi$ is contained in the space $\cS^2_{t}(\R^N)$ of $C^2$-functions $\R^N \to \R$ satisfying that 
  $\partial^{\alpha} \psi \in L^\infty(\R^N,(1+|x|)^{N+2t}dx)$ for $|\alpha| \le 2$. We endow this space with the norm
  $$
  \|\psi\|_{\cS^2_{t}}:= \sup_{|\alpha|\le 2}\|\partial^{\alpha} \psi\|_{L^\infty(\R^N,(1+|x|)^{N+2t}dx)} 
  $$
  Since $t \le \frac{s}{2}\le s-t$, $(-\Delta)^{s-t}$ defines a continuous linear operator
  $$
  \cS^2_{t}(\R^N)  \to  L^\infty(\R^N,(1+|x|)^{N+2t}dx),
  $$
  see e.g. \cite[Lemma 2.1]{FW16}. Using that $C^\infty_c(\R^N)$ is dense in the space $\cS^2_{t}(\R^N)$, we may choose a sequence $\psi_n \in C^\infty_c(\R^N)$ with $\|\psi_n- \phi\|_{\cS^2_{t}} \to 0$ as $n \to \infty$, which, since $g, (-\Delta)^{s-t}g \in L^1_t(\R^N)$ by assumption, implies that
  $$
  \int_{\R^N}g (-\Delta)^{s-t} \psi \,dx = \lim_{n \to \infty}\int_{\R^N}g (-\Delta)^{s-t} \psi_n \,dx= \lim_{n \to \infty}\int_{\R^N}[(-\Delta)^{s-t}g] \psi_n \,dx =   \int_{\R^N} [(-\Delta)^{s-t}g] \psi\,dx.
  $$
  Combining this with \eqref{first-eq-chain-semigroup} gives 
$$
   \int_{\Omega}[(-\Delta)^s g]  \phi\,dx= \int_{\R^N} [(-\Delta)^{s-t}g] \psi\,dx = \int_{\R^N} [(-\Delta)^{s-t}g] (-\Delta)^t \phi\,dx
   = \int_{\Omega} [(-\Delta)^t (-\Delta)^{s-t} g]  \phi\,dx,
 $$
 where we have used (\ref{eq:consistency-definitions}) and the fact that $(-\Delta)^{s-t}g \in L^1_s(\R^N) \cap C^{\alpha-2(s-t)}_{\loc}(\Omega)$ in the final step. Hence, (\ref{eq:distr-sufficient}) holds.
\end{proof} 

We now begin with the main part of the proof of Theorem~\ref{thm1.3}. For this, from now on let
\begin{equation}
  \label{eq:main-assumption-g-differentiability}
g\in L^{\infty}(\RR^N) \cap C^{\alpha}(U)\quad\text{for some $\alpha\in(0,1]$,}
\end{equation}
where $U$ is any open neighborhood of $\Omega$ as fixed in \eqref{omega beta}. As a consequence, for $s\in(0,\frac\alpha{2})$, $(-\Delta)^{s}g$ is defined as a distribution on $\R^N$ and pointwisely as a function in $C^{\alpha-2s}_{\loc}(U)$, cf. Lemma \ref{observation-pointwise}.
Throughout the proof of Theorem~\ref{thm1.3}, we will need to consider different cases distinguished by properties of \begin{equation}
  \label{eq:special-caseb}
f_s:=(-\Delta)^sg.
\end{equation}
As a rather direct consequence of Lemma~\ref{extension-CW-19} and Proposition~\ref{semigroup-general} we have the following.   

\begin{lemma}
\label{diff-express-f-s}
Suppose that 
\begin{equation}
 \label{eq:special-case}
f_s = (-\Delta)^sg \in L^1_{0}(\R^N) 
\end{equation}
for some $s \in (0,\frac{\alpha}{2})$. Then we have 
  \begin{equation*}
    \label{diff-f-s}
  \lim_{\sigma \to 0^+} \Bigl\|\frac{f_{s+\sigma}-f_s}{\sigma}-\loglap f_s\Bigr\|_{L^\infty(\Omega)} = 0.
  \end{equation*}
\end{lemma}

\begin{proof}
The assumption~(\ref{eq:special-case}) implies that $f_s \in L^1_{\sigma}(\R^N)$ for every $\sigma>0$, and therefore Proposition~\ref{semigroup-general} gives 
$$
  \frac{f_{s+\sigma}-f_s}{\sigma} = \frac{(-\Delta)^{\sigma} - \id}{\sigma}f_s \qquad \text{for $\sigma \in (0,\frac{\alpha}{2}-s)$.}
$$
Since $f_s \in C^{\alpha-2s}(\Omega) \cap L^1_0(\R^N)$, the claim now follows from Lemma~\ref{extension-CW-19}. 
\end{proof}

Next, we add the following complement of Lemma~\ref{diff-express-f-s} under a different additional assumption on $g$. 

\begin{lemma}
\label{diff-express-f-s-disjoint-support}
Suppose that
\begin{equation}
  \label{eq:special-case-3}
\text{$g \equiv 0$ in a neighborhood $U$ of $\Omega$,}   
\end{equation}
and let $f_s$ be given by (\ref{eq:special-caseb}) for $s \in (0,\frac{\alpha}{2})$. Then we have 
  \begin{equation*}
    \label{diff-f-s-general}
  \lim_{\sigma \to 0^+} \Bigl\|\frac{f_{s+\sigma}-f_s}{\sigma}- (-\Delta)^{s+\mathrm{Log}} g\Bigr\|_{L^\infty(\Omega)} = 0 \qquad \text{for $s \in (0,\frac{\alpha}{2})$}  
  \end{equation*}
with the fractional-logarithmic operator $(-\Delta)^{s+\mathrm{Log}}$ defined in (\ref{fractional-logarithmic-def}). Moreover, the map 
\begin{equation}
  \label{eq:add-continuity-map}
(0,\frac{\alpha}{2}) \to L^\infty(\Omega), \qquad s \mapsto (-\Delta)^{s+\mathrm{Log}} g  
\end{equation}
is continuous.
\end{lemma}

\begin{proof}
A direct computation gives
  $$
  \Bigl[\frac{f_{s+\sigma}-f_s}{\sigma}- (-\Delta)^{s+\mathrm{Log}} g \Bigr](x)= \int_{\R^N \setminus U}\ell_{\sigma}(x,y)g(y)\,dy\qquad \text{for $x \in \Omega$}
  $$
  with
  \begin{align*}
    \ell_{\sigma}(x,y) &= \frac{1}{\sigma}\Bigl(\frac{c_{N,s+\sigma}}{|x-y|^{N+2s+2\sigma}}-\frac{c_{N,s}}{|x-y|^{N+2s}}\Bigr)-  \frac{1}{|x-y|^{N+2s}}\Bigl(b_{N,s}- c_{N,s}\ln |x-y|\Bigr)\\
  &= |x-y|^{-N-2s}\Bigl[\frac{1}{\sigma}\Bigl(\frac{c_{N,s+\sigma}}{|x-y|^{2\sigma}}-c_{N,s}\Bigr)- \Bigl(b_{N,s}- c_{N,s}\ln |x-y|\Bigr)\Bigr]\\ 
    &= |x-y|^{-N-2s}\Bigl[\ell_\sigma^1(x,y)+ \ell_\sigma^2(x,y)\Bigr]
  \end{align*}
  with
  $$
  \ell_\sigma^1(x,y) = \frac{c_{N,s+\sigma}-c_{N,s}}{\sigma}|x-y|^{-2\sigma}- b_{N,s}
  $$
  and
  $$
  \ell_\sigma^2(x,y) = c_{N,s} \biggl[\frac{1}{\sigma}\Bigl(\frac{1}{|x-y|^{2\sigma}}-1\Bigr)+\ln |x-y|\biggr].
  $$
  Consequently,
  \begin{equation}
    \label{cks-g-l-infty-est}
  \Bigl|\frac{f_{s+\sigma}-f_s}{\sigma}- (-\Delta)^{s+\mathrm{Log}}g\Bigr|(x) \le \|g\|_{L^\infty(\R^N \setminus N)} \int_{\R^N \setminus U}|\ell_{\sigma}(x,y)|\,dy.
  \end{equation}
  Next, we let $\eps>0$, and we choose $R>1$ sufficiently large with $\Omega \subset U \subset B_{R/2}$ and
  $$
  \int_{\R^N \setminus B_R}|y|^{-N-\frac{s}{2}}\,dy < \eps.
  $$
  Then we have $|x-y| \ge |y|-|x| \ge \frac{|y|}{2}$ for $x \in \Omega$, $y \in \R^N \setminus B_R$ and therefore    
  \begin{align*}
    \int_{\R^N \setminus B_R}|x-y|^{-N-2s}|\ell^1_{\sigma}(x,y)|\,dy &\le \Bigl(\Bigl|\frac{c_{N,s+\sigma}-c_{N,s}}{\sigma}\Bigr|+|b_{N,s}|\Bigr) \int_{\R^N \setminus B_R}|x-y|^{-N-2s}\,dy\\
                                                                     &\le \Bigl(\Bigl|\frac{c_{N,s+\sigma}-c_{N,s}}{\sigma}\Bigr|+|b_{N,s}|\Bigr)2^{N+2s} \int_{\R^N \setminus B_R}|y|^{-N-2s}\,dy\\
        &\le \eps \Bigl(\Bigl|\frac{c_{N,s+\sigma}-c_{N,s}}{\sigma}\Bigr|+|b_{N,s}|\Bigr)2^{N+2s} \le \kappa_1(N,s)\eps
  \end{align*}
  for all $\sigma \in [0,1-s]$, $x \in \Omega$ with a constant $\kappa_1(N,s)>0$. Moreover, using the estimate  (see \cite[Lemma 2.1]{JSW20})
  $$
  \Bigl|\frac{1}{|x-y|^{2\sigma}}-1\Bigr| \le \frac{2 \sigma}{s}\Bigl(|x-y|^{-2\sigma-s}+ |x-y|^{s}\Bigr) \le \frac{4 \sigma}{s}|x-y|^{s} \quad \text{for $x \in \Omega$, $y \in \R^N \setminus B_R$,}  
  $$
we find that
  \begin{align*}
    &\int_{\R^N \setminus B_R}|x-y|^{-N-2s}|\ell^2_{\sigma}(x,y)|\,dy \le \frac{4c_{N,s}}{s}\int_{\R^N \setminus B_R}|x-y|^{-N-s}(1+\ln |x-y|)dy\\
    &\le \frac{4c_{N,s}2^{N+s}}{s}\int_{\R^N \setminus B_R}|y|^{-N-s}(1+ \ln (R+|y|))dy\le \kappa_2(N,s)
\int_{\R^N \setminus B_R}|y|^{-N-\frac{s}{2}}dy \le \kappa_2(N,s) \eps
  \end{align*}
  for all $x \in \Omega$  with a constant $\kappa_2(N,s)>0$.
  Consequently,
  \begin{equation}
    \label{eq:l-sigma-outer-ball-est}
    \int_{\R^N \setminus B_R}|\ell_{\sigma}(x,y)|\,dy \le (\kappa_1(N,s) + \kappa_2(N,s)) \eps \qquad \text{for all $\sigma \in [0,1-s]$, $x \in \Omega$.}
  \end{equation}
  In addition, since with $\tau:=\text{dist} (\Omega, \R^N \setminus U)>0$ we have
  $$
  \tau \le |x-y| \le 2 R \qquad \text{for $x \in \Omega$, $y \in B_R \setminus U$,}
  $$
  we have
$$
\ell_\sigma^1(x,y) \to 0 \quad \text{and}\quad \ell_\sigma^{2}(x,y) \to 0 \qquad \text{as $\sigma \to 0^+$ uniformly in $x \in \Omega$, $y \in B_R \setminus U$,}
$$
which implies that
  \begin{equation}
    \label{eq:l-sigma-inner-ball-est}
    \int_{B_R \setminus U}|\ell_{\sigma}(x,y)|\,dy  \to 0 \qquad \text{as $\sigma \to 0^+$ uniformly in $x \in \Omega$.}
  \end{equation}
  Combining (\ref{eq:l-sigma-outer-ball-est}) and (\ref{eq:l-sigma-inner-ball-est}) yields
$$ 
   \limsup_{\sigma \to 0^+}\sup_{x \in \Omega} \int_{\R^N \setminus N}|\ell_{\sigma}(x,y)|\,dy \le (\kappa_1(N,s) + \kappa_2(N,s)) \eps
   \qquad \text{for all $\eps>0$}
$$
and therefore
  \begin{equation}
   \label{eq:l-sigma-combined-lim}
   \lim_{\sigma \to 0^+}\sup_{x \in \Omega} \int_{\R^N \setminus U}|\ell_{\sigma}(x,y)|\,dy =0.
\end{equation}
The claim \ref{diff-f-s-general} now follows by combining (\ref{cks-g-l-infty-est}) and (\ref{eq:l-sigma-combined-lim}).\\
The continuity of the map in~(\ref{eq:add-continuity-map}) follows by very similar estimates, so we skip the details.
    \end{proof}

Next, let $u_s$ denote the unique solution of (\ref{main eq}) for $s \in (0,\frac{\alpha}{2})$. To study the differentiability of the maps $s \mapsto u_s$, we can equivalently consider the function 
\begin{equation}
\label{v-s-u-s}  
    v_s:=g-u_s \in L^\infty(\R^N)
\end{equation}
for $s \in (0,\frac{\alpha}{2})$, which is the unique solution of the problem 
\begin{align*}
 (-\Delta)^s v_s=f_s\,\,\, &\text{in}\,\,\, \Omega,\\
    v_s=0\,\,\, &\text{in}\,\,\, \RR^N\setminus \Omega.    
\end{align*}
Hence, we have $v_s = \cGs_s f_s$ with the Green operator $\cGs$ defined in (\ref{eq:G-op-definition}), and we can apply the results from \cite{JSW25}
to the map $s \mapsto v_s$ in the following.

\begin{proposition}\label{proof diff}
Suppose that (\ref{eq:special-case}) holds for some $s \in (0,\frac{\alpha}{2})$. Then we have
    \begin{equation}\label{claim}
        \lim_{\sigma \to 0^+} \frac{\cGs_{s+\sigma}f_{s+\sigma}-\cGs_{s} f_s }{\sigma} = \cGs_s (w_s- \cPs^{\! \! c} f_s)\qquad \text{in $L^\infty(\Omega)$}
    \end{equation}
with    
\begin{equation}
  \label{eq:w-s-simple}
w_s = L_{\Delta}(1_{\R^N \setminus \Omega}f_s) \qquad \text{in $\Omega$.}  
\end{equation}
\end{proposition}

\begin{proof}
To prove (\ref{claim}), we note that 
\begin{align*}
&\Bigl\| \frac{\cGs_{s+\sigma}f_{s+\sigma}-\cGs_{s}f_s}{\sigma}- \cGs_s (w_s- \cPs^{\! \! c} f_s) \|_{L^\infty(\Omega)}= \Bigl\| \frac{\cGs_{s+\sigma}f_{s+\sigma}-\cGs_{s}f_s}{\sigma}- \cGs_{s} (L_{\Delta} f_s+ \cDs f_s) \|_{L^\infty(\Omega)}\\ \\
&\leq \Bigl\| \frac{(\cGs_{s+\sigma}-\cGs_{s})f_s}{\sigma} - \cGs_{s} \cDs_sf_s\Bigr\|_{L^\infty(\Omega)}+\Bigl\|\cGs_{s+\sigma}\left(\frac{f_{s+ \sigma} -f_s }{\sigma}\right)-\cGs_s L_{\Delta} f_s  \Bigl\|_{L^{\infty}(\Omega)}:=\mathrm I_s +\mathrm {II}_s
    \end{align*}
with 
    \begin{align*}
     \cDs_s f_s = - L_{\Delta} 1_{\Omega}f_s - \cPs^{\! \! c} f_s.    
    \end{align*}
Since 
\begin{align*}
\frac{f_{s+ \sigma} -f_s }{\sigma}  \to L_{\Delta}f_s\qquad \text{as $\sigma \to 0^+$ uniformly in $\Omega$}
\end{align*}
by Lemma~\ref{diff-express-f-s}, it follows from \cref{lemma:continuity-smooth} that 
\begin{align}\label{EQ2.3}
\mathrm{II}_s=\Bigl\|\cGs_{s+\sigma}\left(\frac{f_{s+ \sigma} -f_s }{\sigma}\right)-\cGs_s L_{\Delta}f_s \Bigr\|_{L^{\infty}(\Omega)} \to 0\,\,\, \text{as }\,\,\, \sigma \to 0.
\end{align}
On the other hand, by \cite[Theorem 1.2]{JSW25} and since $f_s\in L^\infty(\Omega)$ by the assumptions on $g$, we have
\begin{align}\label{EQ2.4}
\mathrm{I}_s=\Bigl\|\frac{(\cGs_{s+\sigma}-\cGs_{s})f_s}{\sigma} -\cDs_s f_s\Bigr\|_{L^{\infty}(\Omega)} \to 0 \,\,\, \text{as}\,\,\, \sigma \to 0.    
\end{align}
Combining \eqref{EQ2.3}-\eqref{EQ2.4}, we conclude \eqref{claim}, as claimed.
\end{proof}

\begin{remark}
  \label{rem-special-case-2}
In the special case where   
\begin{equation}
  \label{eq:special-case-2}
g \in C^\alpha_c(\R^N),
\end{equation}
we have
\begin{equation*}
  \label{remainder-C-alpha-minus}
f_s = (-\Delta)^s g \in  C^{\alpha-2s}(\R^N) \cap L^1_0(\R^N)\quad \text{for $s \in (0,\frac{\alpha}{2})$,}  
\end{equation*}
so (\ref{eq:special-case}) holds for all $s \in (0,\frac{\alpha}{2})$. Moreover,  using (\ref{eq:pointwise-identity-fractional-logarithmic}), we find that the function $w_s$ in (\ref{eq:w-s-simple}) can be written as
$$
w_s = L_\Delta \bigl(1_{\R^N \setminus \Omega} f_s \bigr) = L_\Delta f_s -  L_\Delta\bigl(1_{\Omega} f_s \bigr)=(-\Delta)^{s+\mathrm{Log}}g-  L_\Delta\bigl(1_{\Omega} f_s \bigr)\qquad \text{in $\Omega$}.
$$
In the next proposition, we show that the same holds under assumption \eqref{eq:special-case-3} for $g$.
\end{remark}

\begin{proposition}\label{proof diff-case-3}
  Assume $g$ is as in \eqref{eq:special-case-3} and $f_s$ as in \eqref{eq:special-caseb}. Then for $s \in (0,\frac{\alpha}{2})$ we have
    \begin{equation}\label{thm1:eq2-section-case-3}
        \lim_{\sigma \to 0^+} \frac{\cGs_{s+\sigma}f_{s+\sigma}-\cGs_{s} f_s }{\sigma} = \cGs_s (w_s- \cPs^{\! \! c} f_s)\qquad \text{in $L^\infty(\Omega)$}
    \end{equation}
with    
\begin{equation*}
  \label{eq:w-s-simple-case-3}
w_s = (-\Delta)^{s+\mathrm{Log}}g - L_{\Delta}\bigl(1_{\Omega}f_s\bigr) \qquad \text{in $\Omega$.}  
\end{equation*}
\end{proposition}

\begin{proof}
  Following the proof of Proposition~\ref{proof diff} line by line and using Lemma~\ref{diff-express-f-s-disjoint-support} in place of Lemma~\ref{diff-express-f-s}, we obtain \eqref{thm1:eq2-section-case-3} for the right-derivative of $\cGs_s f_s$.
\end{proof}

We may now complete the proof of Theorem~\ref{thm1.3}.

\begin{proof}[Proof of Theorem \ref{thm1.3} completed]
  Using assumption (\ref{eq:main-assumption-g-differentiability}), we may write
  $g = g_1 + g_2$ with $g_1 \in C^\alpha_c(\R^N)$ and $g_2 \in L^\infty(\R^N)$ with $g_2 \equiv 0$ in a neighborhood of $\Omega$. Then, the family of solutions $u_s$ of (\ref{main eq}) for $g$ writes as $u_s = u_s^1 + u_s^2$, where $u_s^1$ resp. $u_s^2$ are the solutions of (\ref{main eq}) with $g$ replaced by $g_1,g_2$, respectively. Hence, combining Propositions~\ref{proof diff},~\ref{proof diff-case-3} and Remark~\ref{rem-special-case-2}, we obtain that 
    \begin{equation}\label{thm1:eq2-section-general-case}
        \lim_{\sigma \to 0^+} \frac{\cGs_{s+\sigma}f_{s+\sigma}-\cGs_{s} f_s }{\sigma} = \cGs_s h_s  \qquad \text{in $L^\infty(\Omega)$ for $s \in (0,\frac{\alpha}{2})$.}
    \end{equation}
    with
    $$
    h_s := (-\Delta)^{s+\mathrm{Log}}g - L_{\Delta}[1_{\Omega}f_s]- \cPs^{\! \! c} f_s.    $$
    To deduce from (\ref{thm1:eq2-section-general-case}) that the map
    $$
    (0,\frac{\alpha}{2}) \to L^\infty(\Omega), \qquad s \mapsto \cGs_s f_s
    $$
    is of class $C^1$ with $\partial_s (\cGs f_s) = \cGs h_s$, it suffices, by \cite[Lemma 6.6]{JSW20}, to show that the maps
    $$
    (0,\frac{\alpha}{2}) \to L^\infty(\Omega), \quad s \mapsto \cGs_s f_s,\quad  
    s \mapsto \cGs_s h_s
    $$
    are continuous. The continuity of the map $s \mapsto \cGs_s f_s$ follows immediately from Lemma~\ref{lemma:continuity-smooth} and the continuity of the map
    $$
    (0,\frac{\alpha}{2}) \to L^\infty(\Omega), \qquad   s \mapsto f_s \big|_\Omega =  [(-\Delta)^s g]\big|_{\Omega}.
    $$
    To see the continuity of the map $    (0,\frac{\alpha}{2}) \to L^\infty(\Omega),\; s \mapsto \cGs_s h_s$, we first note that the map
    $$
    (0,\frac{\alpha}{2}) \to L^\infty(\Omega), \qquad   s \mapsto (-\Delta)^{s+\mathrm{Log}}g
    $$
    is continuous by Lemma~\ref{diff-express-f-s-disjoint-support}, and therefore the map $s \mapsto \cGs_s (-\Delta)^{s+\mathrm{Log}}g$ is continuous as a map to $L^\infty(\Omega)$ by Lemma~\ref{lemma:continuity-smooth}.
    Next, we fix a compact subinterval $[c,d] \subset (0,\frac{\alpha}{2})$, and we claim that the map
    \begin{equation}
      \label{eq:second-summand}
  [c,d] \to L^\infty(\Omega), \qquad s \mapsto \cGs_s  \Bigl(L_{\Delta}1_{\Omega}f_s\Bigr) 
    \end{equation}
  is also continuous. For this we first note that, by (\ref{eq:main-assumption-g-differentiability}) and a straightforward computation, the map
  $$
  [c,d] \to C^\eps(\overline \Omega), \qquad s \mapsto f_s =(-\Delta)^s g
  $$
  is continuous for $0 < \eps < \frac{\alpha}{2}-d$. Moreover the linear map
  $C^\eps(\overline \Omega) \to L^\infty_{\loc}(\Omega),\;   \tilde f \mapsto L_{\Delta}(1_\Omega \tilde f)$ is continuous\footnote{Here $1_\Omega \tilde f$ denotes the trivial extension of $\tilde f$ to $\R^N$}, and this gives the continuity of the map  
  $$
  [c,d] \to L^\infty_{\loc}(\Omega), \qquad s \mapsto L_{\Delta}(1_\Omega f_s).
  $$
  Moreover, using \eqref{alt_integral rep_log} and (\ref{def-h-Omega}), we estimate that 
  \begin{align*}
    |\bigl(L_{\Delta}(1_\Omega f_s)\bigr)(x)| &\le c_N \|f_s\|_{C^\eps(\overline{\Omega})} \int_{\Omega}|x-y|^{\eps-N}dy + \bigl|h_\Omega(x) + \rho_N\bigr| |f_s(x)|\\
                                              &\le C_1 \Log (1+ \frac{1}{\delta_\Omega(x)}) \le C_2 \delta_\Omega^{-s}\quad \text{for $x \in \Omega$}
  \end{align*}
  with constants $C_1,C_2>0$ independent of $s \in [c,d]$. Consequently, the map
  in (\ref{eq:second-summand}) is continuous by Lemma~\ref{lemma:continuity-smooth2app} from the appendix. Finally, we claim that the map
      \begin{equation}
      \label{eq:third-summand}
  [c,d] \to L^\infty(\Omega), \qquad s \mapsto \cGs_s  (\cPs^{\! \! c} f_s)
    \end{equation}
    is also continuous. To see this, we first note that the complementary Poisson kernel in (\ref{eq:complementary-poisson-def}) satisfies 
  \begin{equation}
    \label{eq:complementary-poisson-est}
  \|P_s^{\,c}(x,\cdot)\|_{L^1(\Omega)}\le C \delta_{\Omega}^{-s}(x) \qquad \text{for $x \in \Omega$, $s \in [c,d]$}  
  \end{equation}
  with a constant $C$ depending only on $N,c$, and $d$ (see the proof of \cite[Lemma 2.3]{JSW25}), which implies that
  \begin{align*}
    |  (\cPs^{\!\! c} f_s)(x)| &\le c_N C \|f_s\|_{L^\infty(\Omega)} \delta_\Omega^{-s}\quad \text{for $x \in \Omega$.}
  \end{align*}
  Combining this estimate with Lemmas~\ref{lemma:continuity-smooth2app},~\ref{lemma:continuity-general} and Remark~\ref{lemma:continuity-general-complement-Poisson} from the appendix, we deduce the continuity of the map in (\ref{eq:third-summand}). We thus conclude that the map
  $$
  (0,\frac{\alpha}{2}) \to L^\infty(\Omega), \quad s \mapsto \cGs_s h_s
 $$
 is continuous. As remarked above, it now follows from \cite[Lemma 6.6]{JSW20} that the map
 $$
    (0,\frac{\alpha}{2}) \to L^\infty(\Omega), \qquad s \mapsto \cGs_s f_s
    $$
    is of class $C^1$ with $\partial_s (\cGs_s f_s) = \cGs_s h_s$. Consequently, since $u_s = g - v_s = g- \cGs_s f_s$ by (\ref{v-s-u-s}), also the map
 $$
    (0,\frac{\alpha}{2}) \to L^\infty(\Omega), \qquad s \mapsto u_s
    $$
    is of class $C^1$ with $\partial_s u_s = -\cGs_s h_s$ for $s \in (0,\frac{\alpha}{2})$, which means that $\partial_s u_s$ solves (\ref{general-deriv-formula}). Finally, the representation~(\ref{general-deriv-formula-special-case}) follows from Proposition~\ref{proof diff}. The proof is thus finished.
\end{proof}

We finish this section by completing the

\begin{proof}[Proof of Corollary~\ref{cor-thm-1-3}]
Suppose that $g \in L^{\infty}(\RR^N) \cap C^{\alpha}(U)$ for some $\alpha \in(0,1]$ and with $f_s:= (-\Delta)^s g \in L^1_0(\R^N)$ and $f_s \geq 0$ in $\Omega$ for $s \in (0,\frac{\alpha}{2})$.
By (\ref{general-deriv-formula-special-case}), we then have
$$
\partial_s u_s = \cGs_s \left(\cPs^{\!\! c}\Bigl(f_s \big|_{\Omega}\Bigr)- L_\Delta \bigl(1_{\R^N \setminus \Omega}f_s \bigr)\right)\qquad\text{in $\Omega$.}
$$
Since the complementary Poisson kernel $P_s^{c}(x,z)$ is nonnegative, it follows that 
$$
 \int_{\Omega}P_s^{c}(x,z)f_s(z)\,dz + c_N \int_{\R^N \setminus \Omega}\frac{f_s(y)}{|x-y|^N}\,dy \ge 0 
$$
for all $x \in \Omega$, $s \in (0,\frac{\alpha}{2})$. Consequently, for each $x\in \Omega,$  the map $s \mapsto u_s(x)$ is increasing in $s \in (0, \frac{\alpha}{2})$.
\end{proof}

\appendix

\section{A simple uniform Harnack inequality}\label{short proof of harnack}

Here we give a short proof of the Harnack inequality stated in (\ref{harnack s-0}).

\begin{lemma}[Harnack inequality -- proof of \eqref{harnack s-0}]
Let $s\in(0,1)$, $\Omega\subset \RR^N$ open, and let $u\in L^1_s(\RR^N)$ be a nontrivial, nonnegative function such that $(-\Delta)^su=0$ in $\Omega$. Then, given $x\in\Omega$, $r>0$ such that $B_r(x)\subset \Omega$ we have
$$
2^{-4N-2}\leq \frac{u(x_1)}{u(x_2)}\leq 2^{4N+2}\quad\text{for all $x_1,x_2\in B_{r/4}(x)$.}
$$
\end{lemma}
\begin{proof}
Under the assumptions, it follows that $u|_{\Omega}\in C^{\infty}(\Omega)$ (after passing to a suitable representative of $u$). Hence, $u$ is locally bounded and thus, in particular,
$$
(-\Delta)^su=0 \quad \text{in $B=B_{r/2}(x)$}
$$
and $u$ is bounded in a neighborhood of $B$. Then it follows that for any $z\in B$ we have
$$
u(z)=\frac{\Gamma(\frac{N}{2})}{\pi^{\frac{N}{2}}\Gamma(1-s)\Gamma(s)}\int_{\RR^N\setminus B}\frac{(\frac{r^2}{4}-|z-x|^2)^su(y)}{(|y-x|^2-\frac{r^2}{4})^s|z-y|^N}\, dy.
$$
If $|z-x|\leq \frac{r}4$, $y\in \RR^N\setminus B$ it is not hard to see that
$$
\frac{|x-y|^N}{|z-y|^N}\leq 2^N\Big(1+\frac{|x-z|^N}{|z-y|^N}\Big)\leq 2^N\big(1+2^{N}\big)<8^N
$$
and
$$
\frac{|x-y|^N}{|z-y|^N}\geq \Big(\frac{|z-y|-|x-z|}{|z-y|}\Big)^N=(1-\frac12)^N=2^{-N}.
$$
Thus
\begin{align*}
2^{-N}\frac{3^sr^{2s}}{16^s|x-y|^N}\leq \frac{(\frac{r^2}{4}-|z-x|^2)^s}{|z-y|^N}\leq 8^N\frac{r^{2s}}{4^s|x-y|^N}
\end{align*}
which entails
$$
\frac{3^s}{2^N4^s}u(x)\leq u(z)\leq 8^Nu(x)\quad\text{for all $z\in B_{r/4}(x)$.}
$$
Hence, for $x_1,x_2\in B_{r/4}(x)$ we have
$$
\frac{u(x_1)}{u(x_2)}\leq \frac{8^N2^N4^s}{3^s}\leq 2^{4N+2}\quad\text{and}\quad \frac{u(x_1)}{u(x_2)}\geq \frac{3^s}{2^N4^s 8^N}\geq 2^{-4N-2}
$$
as claimed.
\end{proof}

\section{A weighted limit}\label{weighted limit}
Let $\Omega\subset \RR^N$ open and bounded. Let $s\in(0,1)$ such that $2s<N$ and let $P_s:\Omega\times (\RR^N\setminus\overline{\Omega})\to\RR$ be the fractional Poisson kernel and $G_s:\R^{2N}_{\ast}\to \RR$ the fractional Green kernel. Here, $\Omega$ can be a more general than a $C^2$-domain and we refer to \cite{BKK08} and the references in there for the existence of the fractional Poisson and Green kernel in general nonempty open sets. We note that the formula
\begin{equation}\label{ps-gs relation}
P_s(x,y)=c_{N,s}\int_{\Omega}\frac{G_s(x,z)}{|z-y|^{N+2s}}\, dz\quad\text{for $x\in\Omega$, $y\in \RR^N\setminus\overline{\Omega}$}
\end{equation}
holds.

\begin{remark}
If $\Omega=B_1(0)$, then 
$$
P_s(x,y)=\frac{\Gamma(\frac{N}{2})}{\pi^{\frac{N}{2}}\Gamma(1-s)\Gamma(s)}\frac{(1-|x|^2)^s}{(|y|^2-1)^s|x-y|^N},\quad x\in B_1(0),\, y\in \RR^N\setminus \overline{B_1(0)}
$$
and it is not hard to check that
$$
\lim_{|y|\to\infty} |y|^{2s+N}P_s(x,y)=\frac{\Gamma(\frac{N}{2})}{\pi^{\frac{N}{2}}\Gamma(1-s)\Gamma(s)}(1-|x|^2)^s\quad \text{for}\ x\in B_1(0).
$$
Recall that if $G_s:\RR^N\times \RR^N\to[0,\infty]$ denotes the fractional Green function for $B_1(0)$, then
$$
\int_{\Omega}G_s(x,z)\, dz=\frac{\Gamma(\frac{N}{2})4^{-s}}{\Gamma(1+s)\Gamma(\frac{N}{2}+s)}(1-|x|^2)^s\quad \text{for}\ x\in B_1(0)
$$
and it holds
$$
c_{N,s}\frac{\Gamma(\frac{N}{2})4^{-s}}{\Gamma(1+s)\Gamma(\frac{N}{2}+s)}=\frac{4^ss\Gamma(\frac{N}{2}+s)}{\pi^{\frac{N}{2}}\Gamma(1-s)}\cdot\frac{\Gamma(\frac{N}{2})4^{-s}}{\Gamma(1+s)\Gamma(\frac{N}{2}+s)}=\frac{\Gamma(\frac{N}{2})}{\pi^{\frac{N}{2}}\Gamma(1-s)\Gamma(s)}.
$$
\end{remark}

In spirit of the above remark we next show
\begin{theorem}\label{a limit}
Let $\Omega\subset \RR^N$ be an open bounded set. Then we have
$$
\lim_{|y|\to \infty}\frac{|y|^{2s+N}}{c_{N,s}}P_s(x,y)=\int_{\Omega}G_s(x,z)\, dz \qquad \text{uniformly in $s\in(0,1)$ and $x\in \Omega$.} 
$$
\end{theorem}
\begin{proof}
By \eqref{ps-gs relation} we have, for $x\in \Omega$, $y\in \RR^N\setminus \overline{\Omega}$,
\begin{equation}\label{uniform prove}
\Bigg|\frac{|y|^{2s+N}}{c_{N,s}}P_s(x,y)-\int_{\Omega}G_s(x,z)\, dz\Bigg| \le \int_{\Omega}\Big|\frac{|y| ^{2s+N}}{|z-y|^{2s+N}}-1\Big| G_s(x,z)\, dz.
\end{equation}
Let $\epsilon>0$. Note that since
$$
\lim_{|y|\to\infty} \frac{|y|}{|z-y|}=1 \qquad \text{uniformly in $z \in \Omega$,}
$$
we find $R>0$ large enough such that  
$$
\frac{|y|}{|z-y|}\leq 1+\epsilon.
$$
Hence, from \eqref{uniform prove} we have for $|y|>R$
\begin{align*}
&\Bigg|\frac{|y|^{2s+N}}{c_{N,s}}P_s(x,y)-\int_{\Omega}G_s(x,z)\, dz\Bigg| \leq \int_{\Omega} \Big( (1+\epsilon)^{2s+N}-1\Big)G_s(x,z)\, dz\\
  &\leq (2s+N)\epsilon 2^{2s+N-1}\int_{\Omega} G_s(x,z)\, dz \leq \epsilon(2+N) 2^{1+N}\int_{\Omega} F_s(z-x)\, dz\\
  &= \epsilon(2+N) 2^{1+N}\kappa_{N,s}\int_{\Omega}|z-x|^{2s-N}\, dz \le \epsilon(2+N) 2^{1+N}\kappa_{N,s}\int_{B_{r_\Omega}(x)}|z-x|^{2s-N}\, dz \\
  &= \epsilon C_s(\Omega) \qquad \qquad \text{with}\qquad C_s(\Omega):= \omega_{N-1}(2+N) 2^{1+N}\frac{\kappa_{N,s}}{2s}r_\Omega^{2s}, 
\end{align*}
where again $r_\Omega:= \diam \Omega +1$. By (\ref{eq:kappa-N-s-limit}), $C_s(\Omega)$ remains uniformly bounded for $s  \in (0,1)$.
Since $\eps>0$ was chosen arbitrarily, we deduce that 
$$
\lim_{|y|\to\infty} \Bigg|\frac{|y|^{2s+N}}{c_{N,s}}P_s(x,y)-\int_{\Omega}G_s(x,z)\, dz\Bigg|=0
$$
uniformly in $s\in(0,1)$ and $x\in \Omega$, as claimed.
\end{proof}

\section{A continuity statement involving the Green operator}\label{weighted limit2}

Let $\Omega\subset\R^N$ be an open bounded set. Recalling the definition of the Green operator in \eqref{eq:G-op-definition} and the Riesz operator we present here the proof of the following continuity statement, which is contained within the proof of \cite[Theorem 6.1]{JSW20} and for the reader's convenience we provide it in the formulation as used in the present manuscript.
\begin{lemma}[cf. Lemma 5.2, \cite{JSW20}]\label{lemma:continuity-smooth2app}
	Let $a,b \in [0,1]$, $a<b$, and let $(a,b)\to L^{\infty}_{loc}(\Omega)$, $s\mapsto f_s$ be a continuous curve with 
  \begin{equation}
    \label{lemma:continuity-smooth2app-add-0}
|f_s(x)|\leq C\delta_{\Omega}^{-s}(x) \qquad \text{for $x \in \Omega$, $s \in (a,b)$ with a constant $C>0$.}
\end{equation}
Then the map
$$
(a,b)\to L^{\infty}(\Omega),\quad s\to \cGs_sf_s
$$
is continuous, where $\cGs_{s}$ is the Green operator defined in \eqref{eq:G-op-definition}.
\end{lemma}
\begin{proof}
  Let $s_0\in(a,b)$, fix $\epsilon\in(0,\min\{a,b-s_0\})$ and choose a sufficiently large compact subset $K \subset \Omega$ with
  \begin{equation}
    \label{lemma:continuity-smooth2app-add-1}
    \delta_\Omega^{\epsilon}(x) \le \frac{\epsilon}{C} \qquad \text{for $x \in \Omega \setminus K$}
\end{equation}
For $s \in (a,b)$, we then write $f_s=f_{1,s}+f_{2,s}$ with $f_{1,s} = f_s 1_K$ and $f_{2,s} = f_s 1_{\Omega \setminus K}$. Then the map
$$
(a,b) \to L^\infty(\Omega), \qquad s \mapsto f_{1,s}
$$
is continuous, and therefore Lemma \ref{lemma:continuity-smooth} gives the continuity of the map
$$
(a,b)\to L^{\infty}(\Omega),\quad s\mapsto \cGs_s f_{1,s}.
$$
Moreover, by~(\ref{lemma:continuity-smooth2app-add-0}) and (\ref{lemma:continuity-smooth2app-add-1}) we have 
  \begin{equation}
    \label{lemma:continuity-smooth2app-add-2}
\|\delta_{\Omega}^{s+\epsilon}f_{2,s}\|_{L^\infty(\Omega)}<\epsilon \qquad \text{for $s \in (a,b)$.}
\end{equation}
Then, for $\sigma\in \R$, $|\sigma|<\min\{a,b-s_0\}-\epsilon$ it follows with a similar estimate to Lemma \ref{int:la} (see \cite[Lemma 2.5]{JSW20}) that for $x\in \Omega$ we have
\begin{align*}
  &\left|\Bigl(\cGs_{{s_0}+\sigma}f_{2,{s_0+\sigma}}-\cGs_{s_0}f_{2,{s_0}}\Bigr)(x)\right|\leq  \left|\bigl(\cGs_{{s_0}+\sigma}f_{2,{s_0+\sigma}}\bigr)(x)\right|+\left|\bigl(\cGs_{s_0}f_{2,{s_0}}\bigr)(x)\right|\\
                                                                                         &\le \bigl(\cFs_{{s_0}+\sigma}|f_{2,{s_0+\sigma}}|\bigr)(x)+\bigl(\cFs_{s_0}|f_{2,{s_0}}|\bigr)(x)\\                                                                  &\le \|\delta_{\Omega}^{s_0+\sigma +\epsilon}f_{2,{s_0+\sigma}}\|_{L^\infty(\Omega)} \Bigl(\cFs_{s_0+\sigma}\delta_{\Omega}^{-s_0-\sigma-\epsilon}\Bigr)(x)+            \|\delta_{\Omega}^{s_0+\epsilon}f_{2,{s_0}}\|_{L^\infty(\Omega)} \Bigl(\cFs_{s_0}\delta_{\Omega}^{-s_0-\epsilon}\Bigr)(x)  \\
&\leq \epsilon \tilde C\left[(\min\{s_0+\sigma,1-s_0-\sigma\}-\epsilon)^{-3}+(\min\{s_0,1-s_0\}-\epsilon)^{-3}\right]
\end{align*}
with a constant $\tilde C>0$. Hence,
\begin{align*}
\limsup_{\sigma\to0}
  \left\|\cGs_{s_0+\sigma}f_{s_0+\sigma}-\cGs_s f_{s_0}\right\|_{L^{\infty}(\Omega)}&\leq \limsup_{\sigma\to 0}\left\|\cGs_{s_0+\sigma}f_{2,s_0+\sigma}-\cGs_s f_{2,s_0}\right\|_{L^{\infty}(\Omega)}\\
  &\leq \epsilon \tilde C(\min\{s,1-s\}-\epsilon)^{-3}.
\end{align*}
Since $\epsilon>0$ was chosen arbitrary, the claim follows.
\end{proof}

We also add the following continuity statement.

\begin{lemma}\label{lemma:continuity-general}
  Let $a,b \in [0,1]$, $a<b$, and let $(a,b)\to L^{\infty}(\Omega)$, $s\mapsto f_s$ be a continuous curve. Moreover, consider a continuous map
  $$
  (a,b) \times \Omega \times \Omega \to \R, \qquad (s,x,z) \to p_s(x,z),
  $$
and suppose that there exists $r>1$ with
  \begin{equation}
    \label{lemma:continuity-general-add-0}
\sup_{x \in K, a<s<b} \|p_s(x,\cdot)\|_{L^r(\Omega)}< \infty \quad \text{for every compact subset $K \subset \Omega$.}
\end{equation}
Then the map
$$
(a,b)\to L^{\infty}_{\loc}(\Omega),\quad s \to \int_{\Omega} p_s(\cdot,z)f_s(z)\,dz
$$
is continuous.    
\end{lemma}

\begin{proof}
  Let $K \subset \R^N$ be a compact set and
  $$
  c_K:= \sup_{x \in K, a<s<b} \|p_s(x,\cdot)\|_{L^r(\Omega)}< \infty.
  $$
  Moreover, let $s \in (a,b)$ and $\sigma \in \R$ with $s+ \sigma \in (a,b)$. Then for $x \in K$ we have
  \begin{align*}
    &\Bigl|\int_{\Omega} p_{s+\sigma}(x,z)f_{s+\sigma}(z)\,dz -\int_{\Omega} p_s(x,z)f_s(z)\,dz \Bigr|\\
    &\le \Bigl|\int_{\Omega} p_{s+\sigma}(x,z)\bigl(f_{s+\sigma}(z)-f_s(z)\bigr)\,dz\Bigr|+ \Bigl|\int_{\Omega} (p_{s+\sigma}(x,z)-p_s(x,z)\bigr)f_s(z)\,dz \Bigr|, 
   \end{align*}
   where
   $$
   \sup_{x \in K} \Bigl|\int_{\Omega} p_{s+\sigma}(x,z)\bigl(f_{s+\sigma}(z)-f_s(z)\bigr)\,dz
   \le \sup_{x \in K} \|\|p_{s+\sigma}(x,\cdot)\|_{L^1(\Omega)} \|f_{s+\sigma}-f_s\|_{L^\infty(\Omega)} \to 0
   $$
   as $\sigma \to 0$ by assumption. Moreover, for every compact subset $\tilde K \subset \R^N$, we have
   $$
   \sup_{x \in K} \Bigl|\int_{\tilde K} (p_{s+\sigma}(x,z)-p_s(x,z)\bigr)f_s(z)\,dz \Bigr| \to 0 \qquad \text{as $\sigma \to 0$}
   $$
   by the locally uniform continuity of the map $(s,x,z) \to p_s(x,z)$. 
   Letting $\eps>0$ and choosing $\tilde K$ with $|\Omega \setminus \tilde K|<\eps$, we then conclude that
   \begin{align*}
   &\limsup_{\sigma \to 0} \Bigl\|\int_{\Omega} p_{s+\sigma}(x,z)f_{s+\sigma}(z)\,dz -\int_{\Omega} p_s(x,z)f_s(z)\,dz \Bigr\|_{L^\infty(K)}\\
     &\le \limsup_{\sigma \to 0}\Bigl\|\int_{\Omega \setminus \tilde K} (p_{s+\sigma}(x,z)-p_s(x,z)\bigr)f_s(z)\,dz \Bigr \|_{L^\infty(K)} \\
     &\le 2 \|f_s\|_{L^\infty(\Omega)} \sup_{s \in (a,b), x \in K}\|p_s(x,\cdot)\|_{L^1(\Omega \setminus \tilde K)}\\ 
     &\le 2 \|f_s\|_{L^\infty(\Omega)}|\Omega \setminus \tilde K|^{\frac{1}{r'}} \sup_{s \in (a,b), x \in K}\|p_s(x,\cdot)\|_{L^r(\Omega)}\\ 
     &\le 2 \eps^{\frac{1}{r'}} c_K \|f_s\|_{L^\infty(\Omega)} 
   \end{align*}
with $r' = \frac{r}{r-1}$. Since $\eps>0$ was chosen arbitrarily, it follows that 
$$
\lim_{\sigma \to 0} \Bigl\|\int_{\Omega} p_{s+\sigma}(\cdot,z)f_{s+\sigma}(z)\,dz -\int_{\Omega} p_s(\cdot,z)f_s(z)\,dz \Bigr\|_{L^\infty(K)}= 0,
$$
which proves the claim of the lemma.  
\end{proof}

\begin{remark}
  \label{lemma:continuity-general-complement-Poisson}
  Lemma~\ref{lemma:continuity-general} applies, in particular, to the family of complementary Poisson kernels $p_s = P_s^c$ defined in (\ref{eq:complementary-poisson-def}). In fact, since
  $$
  \int_{\R^N \setminus \Omega}P_s(z,y)\,dy  = 1 \qquad \text{for every $z \in \Omega$, $s \in (0,1)$,}
  $$
the definition of $P_s^c$ implies that 
  $$
  0 \le P_s^c(x,z) \le c_N \delta_\Omega^{-N}(x)\int_{\R^N \setminus \Omega}P_s(z,y)\,dy \le c_N \delta_\Omega^{-N}(x)\qquad \text{for every $z \in \Omega$, $s \in (0,1)$.}
  $$
and therefore 
  \begin{equation}
    \label{lemma:continuity-general-add-0}
\sup_{x \in K, 0<s<1} \|p_s(x,\cdot)\|_{L^\infty(\Omega)}< \infty \quad \text{for every compact subset $K \subset \Omega$.}
\end{equation}
\end{remark}

\section{Properties of the fractional logarithmic Laplacian}\label{appD} Here we provide a proof of \eqref{eq:pointwise-identity-fractional-logarithmic} for functions $u \in C^{\alpha}_{\loc}(\R^N)$  with $\alpha>2s$, noting that this pointwise identity had already been proved in \cite[Proposition 1.1 and Theorem 1.1]{CH26} for $u\in C^2_c(\RR^N)$.

\begin{lemma}\label{rep c1}
Let $0<s<1$, $\Omega,U\subset \RR^N$ open with $\overline{\Omega}\subset U$, and $u \in C^{\alpha}_{\loc}(U)\cap L^{\infty}(\RR^N)$ with $\alpha >2s.$ Then, for $x \in \Omega$ we have
\begin{align*}
 [(-\Delta)^{s+\mathrm{Log}}u](x)=\frac{d}{ds}[(-\Delta)^su](x)= c_{N,s}\text{P.V.}\int_{\RR^N} \frac{u(x)-u(x+z)}{|z|^{N+2s}}\bigl(-2\ln|z|\bigr)\,dz+b_{N,s}(-\Delta)^s u(x)  
\end{align*}
 where $\frac{d}{ds}c_{N.s}=b_{N, s}.$ Moreover, if $U=\RR^N$, then for each $x \in \RR^N,$ we have
 \begin{align}\label{commute-fractional-log}
 L_{\Delta}(-\Delta)^s u(x)=(-\Delta)^sL_{\Delta}u(x).
 \end{align}
\end{lemma}
\begin{proof} 
Fix $x\in \Omega$ and without loss of generality we assume $B_1(x)\subset \Omega$ (otherwise, we may relabel $u$ by considering $v(x)=u(tx)$). It is enough to consider $0<t<s<\frac{\alpha}{2}<1$ and to compute
\begin{align*}
\frac{(-\Delta)^s u(x)-(-\Delta)^t u(x)}{s-t}\quad \text{as}\quad t\to s.
\end{align*}
Indeed, we have
\begin{align*}
&\frac{(-\Delta)^s u(x)-(-\Delta)^t u(x)}{s-t}\\
&=\frac{c_{N, s}-c_{N, t}}{s-t}\text{P.V.}\int_{\RR^N}\frac{u(x)-u(x+z)}{|z|^{N+2s}}\,dz+c_{N, t}\text{P.V.}\int_{\RR^N}\left(u(x)-u(x+z)\right)\frac{|z|^{-(N+2s)}-|z|^{-(N+2t)}}{s-t}\, dz.
\end{align*}
First, we show that
\begin{align}\label{eq:convergence}
\left\|\,x\mapsto\text{P.V.}\int_{\RR^N}\left(u(x)-u(x+z)\right)\left(\frac{|z|^{-(N+2s)}-|z|^{-(N+2t)}}{s-t}+\frac{2\ln|z|}{|z|^{N+2s}}\right)\, dz\right\|_{L^{\infty}(\Omega)} \to 0\quad \text{as} \quad t\to s.
\end{align}
We note that \[F_t(|z|):=\frac{|z|^{-(N+2s)}-|z|^{-(N+2t)}}{s-t}+\frac{2\ln|z|}{|z|^{N+2s}}\to 0 \quad \text{as}\quad t\to s,\]
and since $u \in C^{\alpha}_{\loc}(U),$ we have
\begin{align*}
    \int_{\varepsilon < |z|< 1}\left(u(x)-u(x+z)\right)F_t(|z|)\, dz \leq \int_{\varepsilon < |z|< 1}|z|^{\alpha}F_t(|z|)\, dz.
\end{align*}
An application of the mean value theorem guarantees that there exists a $\xi\in (t, s)$ such that
\begin{align*}
|F_t(|z|)|\leq G_s(|z|)=2|\ln |z||\left(|z|^{-(N+2s)}+|z|^{-(N+2\xi)}\right).
\end{align*}
Hence, we obtain
\begin{align*}
\int_{\varepsilon<|z|<1}|z|^{\alpha}|F_t(|z|)|\leq 2\int_{\varepsilon<|z|<1}|z|^{\alpha-\kappa-N-2s}+|z|^{\alpha-\kappa-N-2\xi}\, dz<\infty\quad \text{for every small}\quad \varepsilon, \kappa>0,
\end{align*}
since $\xi<s<\frac{\alpha}{2}.$ Thus using Lebesgue dominated convergence theorem, we deduce that 
\begin{align*}
\sup_{x\in V }\text{P.V} \int_{|z|<1}\left(u(x)-u(x+z)\right)\left(\frac{|z|^{-(N+2s)}-|z|^{-(N+2t)}}{s-t}+\frac{2\ln|z|}{|z|^{N+2s}}\right)\, dz \to 0\quad \text{as} \quad t\to s.
\end{align*}
On the other hand, for $|z|\geq 1,$ we have
\begin{align*}
\int_{|z|\geq 1}\left(u(x)-u(x+z)\right)F_t(|z|)\, dz\leq 2\|u\|_{L^{\infty}(\RR^N)}\int_{\RR^N}|\ln|z||\left(|z|^{-(N+2s)}+|z|^{-(N+2\xi)}\right)\, dz< \infty.
\end{align*}
Hence, again by dominated convergence, we conclude
\begin{align*}
    \sup_{x\in \Omega} \int_{|z|\geq 1} \left(u(x)-u(x+z)\right)\left(\frac{|z|^{-(N+2s)}-|z|^{-(N+2t)}}{s-t}+\frac{2\ln|z|}{|z|^{N+2s}}\right)\, dz \to 0\quad \text{as} \quad t\to s.
\end{align*}
Combining these two convergences above, we have \eqref{eq:convergence}. Finally, since  by definition
\begin{align*}
 \frac{c_{N, s}-c_{N, t}}{s-t}\to b_{N, s}\,\,\, \text{and}\,\,\, c_{N, t}\to c_{N, s}\,\,\, \text{as}\,\,\, t\to s,
\end{align*}
we complete the proof the first statement.

The second statement \eqref{commute-fractional-log} is a direct consequence of \cite[Theorem 1.1 (iv)]{CH26} since $C^{\alpha}(\RR^N)\cap L^{\infty}(\RR^N)\subset \mathcal{Z}^{\prime}(\RR^N)$, where $\mathcal{Z}(\RR^N)$ denotes the Lizorkin space on $\RR^N.$ 
\end{proof}

We prove the regularity of the fractional logarithmic Laplacian in the following result, which might be of independent interest.
\begin{lemma}\label{holder for frac log laplacian}
  Let $0<s<1$, $\Omega,U\subset \RR^N$ open with $\overline{\Omega}\subset U$, and $u \in C^{\alpha}_{\loc}(U)\cap L^{\infty}(\RR^N)$ for some $\alpha\in (0,1]$ with $\alpha>2s$. Then $(-\Delta)^{s+\mathrm Log} u\in C^{\alpha-\kappa -2s}(\Omega)$ for any $\kappa>0.$
\end{lemma}
\begin{proof}
 Let $d=\text{dist}(\Omega, \partial U)$ and set $\varepsilon = |x_1-x_2|.$ Since\footnote{Indeed, the statement is also true for any $s>0$, but one has to be more careful due to the higher order finite differences in the numerator. As it is not relevant for our analysis, we use the present formulation for the reader's convenience.} $s<\frac{1}{2}, $ we recall $(-\Delta)^{s+\mathrm Log}u$ without the principal value,
 \begin{align*}
  [(-\Delta)^{s+\mathrm Log}u](x)&= c_{N,s}\int_{\RR^N} \frac{u(x)-u(x+z)}{|z|^{N+2s}}\bigl(-2\ln|z|\bigr)\,dy+b_{N,s}(-\Delta)^s u(x)\\
  &=I^s_{\mathrm Log}u (x)+b_{N, s}(-\Delta)^s u(x)    
 \end{align*}
 Note that the second part of the above expression is $(-\Delta)^s$ and it is known that $(-\Delta)^su \in C^{\alpha-2s}(\Omega).$ Hence, we shall only consider the first part, which involves a logarithmic contribution in the kernel. For any $x_1, x_2 \in \Omega,$ we estimate
\begin{align*}
    |I^s_{\mathrm Log} u(x_1)-I^s_{\mathrm Log} u(x_2)|&\leq 2c_{N, s}\Bigg| \int_{\RR^N}\frac{\left[(u(x_1)-u(x_1+z))-(u(x_2)-u(x_2+z))\right]\ln |z|}{|z|^{N+2s}}\, dz\Bigg|\\
    &\leq 2c_{N, s}\Bigg|\int_{B_{\varepsilon}(0)}... \, dz\Bigg|+\Bigg|\int_{\RR^N\setminus B_{\varepsilon}(0)}...\, dz\Bigg|=2c_{N, s}(\mathrm I+\mathrm {II}).
\end{align*}
\textbf{Estimate of $\mathrm I:$ } Using $u \in C^{\alpha}(U),$ we get
\begin{align*}
    \mathrm I=&\int_{B_{\varepsilon}(0)}\frac{|[(u(x_1)-u(x_1+z))-(u(x_2)-u(x_2+z))]\ln |z||}{|z|^{N+2s}}\, dz\leq [u]_{0, \alpha}\int_{B_{\varepsilon}(0)}\frac{|z|^{\alpha}|\ln |z||}{|z|^{N+2s}}\, dz\\
    &\leq c(\kappa)[u]_{0, \alpha}\int_0^{\varepsilon}\rho^{N-1}\rho^{\alpha-\kappa-N-2s}\, d\rho= \frac{c(\kappa)[u]_{0, \alpha}}{\alpha-\kappa-2s} \varepsilon^{\alpha-\kappa -2s}=\frac{c(\kappa, [u]_{0, \alpha})}{\alpha-\kappa-2s}|x_1-x_2|^{\alpha-\kappa-2s}.
\end{align*}
\textbf{Estimate of $\mathrm {II}:$ } Here we can split the integral as
\begin{align*}
    \mathrm{II}&=\Bigg|\int_{\RR^N\setminus B_{\varepsilon}(0)}\frac{[(u(x_1)-u(x_1+z))-(u(x_2)-u(x_2+z))]\ln |z|}{|z|^{N+2s}}\, dz\Bigg|\\
    &\leq |(u(x_1)-u(x_2)|\int_{\RR^N\setminus B_{\varepsilon}(0)}\frac{|\ln |z||}{|z|^{N+2s}}\, dz+\Bigg|\int_{\RR^N\setminus B_{\varepsilon}(0)}\frac{(u(x_1+z)-u(x_2+z))\ln |z|}{|z|^{N+2s}}\, dz\Bigg|.
\end{align*} 
The first integral can be estimated again using the H\"{o}lder continuity of $u,$ and we obtain
\begin{align*}
 |(u(x_1)-u(x_2)|\int_{\RR^N\setminus B_{\varepsilon}(0)}\frac{|\ln |z||}{|z|^{N+2s}}\, dz\leq [u]_{0, \alpha}|x_1-x_2|^{\alpha} \int_{\varepsilon}^{\infty}|\ln \rho|\rho^{N-1}\rho^{-N-2s}\, d\rho.  
\end{align*}
We note that
\begin{align*}
    \int_{\varepsilon}^{\infty} |\ln \rho|\rho^{-1-2s}\, d\rho&=\int_{\varepsilon}^{1} (-\ln \rho)\rho^{-1-2s}\, d\rho+ \int_{1}^{\infty}  (\ln \rho)\rho^{-1-2s}\, d\rho\\
    &=\varepsilon^{-2s}\left[\frac{\ln (1/\varepsilon)}{2s}-\frac{1}{4s^2}\right]+\frac{1}{2s^2}\leq \varepsilon^{-2s}\left[\frac{\ln (1/\varepsilon)}{2s}+\frac{1}{4s^2}\right]\\
    &\leq \varepsilon^{-\kappa-2s}\left(\frac{1}{2s}+\frac{1}{4s^2}\right).
\end{align*}
Using this in the above estimate, we get
\begin{align*}
|(u(x_1)-u(x_2)|\int_{\RR^N\setminus B_{\varepsilon}(0)}\frac{|\ln |z||}{|z|^{N+2s}}\, dz \leq c(s, [u]_{0, \alpha})|x_1-x_2|^{\alpha-\kappa-2s}.    
\end{align*}

We use a change of variable to estimate the second integral of $\mathrm {II}.$ 
\begin{align*}
&\Bigg|\int_{\RR^N\setminus B_{\varepsilon}(0)}\frac{[u(x_1+z)-u(x_2+z)]\ln |z|}{|z|^{N+2s}}\, dz\bigg|\\
&\le \int_{\varepsilon< |z|\leq \frac{d}{2}}\frac{|u(x_1+z)-u(x_2+z)||\ln |z||}{|z|^{N+2s}}\, dz+\bigg|\int_{|z|>\frac{d}{2}}\frac{[u(x_1+z)-u(x_2+z)]\ln |z|}{|z|^{N+2s}}\, dz\bigg|.
\end{align*}
Note that since $|z|\leq \frac{d}{2},$ for any $x_1, x_2\in \Omega,$ we have that $x_i+z\in U, $ $i=1,2.$ Therefore, using the H\"older continuity of $u, $ we derive
\begin{align*}
 \int_{\varepsilon< |z|\leq \frac{d}{2}}\frac{|u(x_1+z)-u(x_2+z)||\ln |z||}{|z|^{N+2s}}\, dz\leq c(s, [u]_{0, \alpha})|x_1-x_2|^{\alpha-\kappa-2s}.   
\end{align*}
Next, using change of variable in the second integral, we get
\begin{align*}
&\bigg|\int_{|z|>\frac{d}{2}}\frac{[u(x_1+z)-u(x_2+z)]\ln |z|}{|z|^{N+2s}}\, dz\bigg|\\
&=\bigg| \int_{|x_1-y|>\frac{d}{2}}\frac{u(y)\ln |x_1-y|}{|x_1-y|^{N+2s}}\, dy-\int_{|x_2-y|>\frac{d}{2}}\frac{u(y)\ln |x_2-y|}{|x_2-y|^{N+2s}}\, dy\bigg|.
\end{align*}
Let us define the sets
\begin{align*}
    D_1:=\left\{y\;:\; |x_1-y|> \frac{d}{2}, |x_2-y|> \frac{d}{2}\right\},\\
    D_2:=\left\{y\;:\; |x_1-y|> \frac{d}{2}, |x_2-y|\leq \frac{d}{2}\right\},\\
    D_3:=\left\{y\;:\; |x_1-y|\leq \frac{d}{2}, |x_2-y|> \frac{d}{2}\right\},
\end{align*}
and write the above integral in the following form.
\begin{align*}
&\bigg| \int_{|x_1-y|>\frac{d}{2}}\frac{u(y)\ln |x_1-y|}{|x_1-y|^{N+2s}}\, dy-\int_{|x_2-y|>\frac{d}{2}}\frac{u(y)\ln |x_2-y|}{|x_2-y|^{N+2s}}\, dy\bigg|\\
&=\Bigg|\int_{D_1}u(y)(\ln|x_1-y||x_1-y|^{-N-2s}-\ln |x_2-y||x_2-y|^{-N-2s})\, dy\\
&+\int_{D_2}u(y)\ln |x_1-y||x_1-y|^{-N-2s}\, dy-\int_{D_3}u(y)\ln |x_2-y||x_2-y|^{-N-2s}\, dy\Bigg|\\
&\leq||u||_{L^{\infty}(\RR^N)}\int_{D_1}\Big|\ln |x_1-y||x_1-y|^{-N-2s}-\ln |x_2-y||x_2-y|^{-N-2s}\Big|\, dy\\
&+ ||u||_{L^{\infty}(\RR^N)}\int_{D_2}|\ln |x_1-y|||x_1-y|^{-N-2s}\, dy+||u||_{L^{\infty}(\RR^N)}\int_{D_3}|\ln |x_2-y|||x_2-y|^{-N-2s}\, dy.
\end{align*}
We consider the function $h: (d/2, \infty)\to \mathbb{R}$ 
\[h(\rho)=(\ln \rho)\rho^{-m}\quad \text{for}\quad m\geq 1.\] Using mean value theorem, for any $\rho_1, \rho_2\in (d/2, \infty),$ we have
\begin{align*}
    \left|h(\rho_1)-h(\rho_2)\right|\leq \max_{\xi \in [\rho_1, \rho_2]}\left|\xi^{-m-1}(1-m\ln \xi)\right||\rho_1-\rho_2|.
\end{align*}
Since, we have 
\begin{align*}
    \int_{d/2}^{\infty}\rho^{-2(s+1)}|\ln \rho|\, d\rho< \infty,
\end{align*} in particular, taking $m=N+2s,$
we estimate the first term as
\begin{align*}
 &\int_{D_1}\Big|\ln |x_1-y||x_1-y|^{-N-2s}-\ln |x_2-y||x_2-y|^{-N-2s}\Big|\, dy\\
 &\leq |x_1-x_2|\Bigg[\int_{D_1}\left||x_1-y|^{-(N+2s)-1}\left(1-(N+2s)\ln |x_1-y|\right)\right|\\
 &+\left||x_2-y|^{-(N+2s)-1}\left(1-(N+2s)\ln |x_2-y|\right)\right|\,dy\Bigg]\\
 &\leq c(d/2, N, s)|x_1-x_2|.
\end{align*}
To estimate the second and third integral, we first observe that for $y \in D_2,$
\begin{align*}
\frac{d}{2}< |x_1-y|\leq |x_1-x_2|+ \frac{d}{2}.
\end{align*}
This implies $D_2 \subset \left\{y : \frac{d}{2}< |x_1-y|\leq |x_1-x_2|+ \frac{d}{2}\right\}:=\widetilde D_2.$ Hence, we have
\begin{align*}
  \int_{D_2}|\ln |x_1-y|||x_1-y|^{-N-2s}\, dy &\leq \int_{\widetilde D_2} |\ln |x_1-y|||x_1-y|^{-N-2s}\, dy\\
  &=\int_{d/2}^{|x_1-x_2|+d/2}|\ln \rho|\rho^{-1-2s}\,d\rho\\
  &\leq \max\left\{\ln (2/d), \ln (|x_1-x_2|+d/2)\right\}(\frac{d}{2})^{-(1+2s)}|x_1-x_2|\\
  &\leq \max\left\{\ln (2/d), \ln (\text{diam}\, \Omega+d/2)\right\}(\frac{d}{2})^{-(1+2s)}|x_1-x_2|^{\alpha-\kappa-2s}.
\end{align*}
Similarly, we note that 
\[D_3 \subset \left\{y : \frac{d}{2}< |x_2-y|\leq |x_1-x_2|+ \frac{d}{2} \right\}:=\widetilde D_3\] and proceeding as before, we get the same bound.
\begin{align*}
\int_{D_3}|\ln |x_2-y|||x_2-y|^{-N-2s}\, dy &\leq \int_{\widetilde D_3} |\ln |x_2-y|||x_2-y|^{-N-2s}\, dy\\
&\leq  \max\left\{\ln (2/d), \ln (\text{diam}\, \Omega+d/2)\right\}(\frac{d}{2})^{-(1+2s)}|x_1-x_2|^{\alpha-\kappa-2s}.
\end{align*}
Combining all the estimates above, for any $x_1, x_2 \in \Omega,$ we obtain 
\begin{align*}
|(-\Delta)^{s+\mathrm Log} u(x_1)-(-\Delta)^{s+\mathrm Log}u(x_2)| \leq C(\kappa, d, s, N,  ||u||_{L^{\infty}(\RR^N)}, [u]_{0, \alpha})|x_1-x_2|^{\alpha-\kappa-2s}.
\end{align*} This completes the proof. 
\end{proof}

\noindent\textbf{Acknowledgment.}
Abhrojyoti Sen is supported by a postdoctoral research fellowship from the Alexander von Humboldt Foundation, Germany. The author also thanks Goethe University Frankfurt for its support and research environment.

\bigskip
\noindent\textbf{Conflict of interest.} On behalf of all authors, the corresponding author states that there is no conflict of interest.

\bigskip
\noindent\textbf{Data availability.} This manuscript does not have associated data.

\end{document}